\documentclass[english]{amsart}
\usepackage{libertine}
\usepackage[T1]{fontenc}
\usepackage[utf8]{inputenc}
\usepackage{babel}
\usepackage{url}
\usepackage{amsmath}
\usepackage{amsthm}
\usepackage{amssymb}
\usepackage{setspace}
\usepackage[all]{xy}
\usepackage[unicode=true,pdfusetitle,
 bookmarks=true,bookmarksnumbered=false,bookmarksopen=false,
 breaklinks=false,pdfborder={0 0 0},pdfborderstyle={},backref=false,colorlinks=false]
 {hyperref}
\usepackage{breakurl}
\usepackage{bbm}

\makeatletter
\numberwithin{equation}{section}
\numberwithin{figure}{section}
\theoremstyle{plain}
\newtheorem{thm}{\protect\theoremname}[section]
  \theoremstyle{plain}
  \newtheorem{lem}[thm]{\protect\lemmaname}
  \theoremstyle{definition}
  \newtheorem{defn}[thm]{\protect\definitionname}
  \theoremstyle{remark}
  \newtheorem*{rem*}{\protect\remarkname}
  \theoremstyle{remark}
  \newtheorem{rem}[thm]{\protect\remarkname}
  \theoremstyle{definition}
  
  \theoremstyle{plain}
  \newtheorem{cor}[thm]{\protect\corollaryname}
  \theoremstyle{definition}
  \newtheorem*{example*}{\protect\examplename}
  \theoremstyle{plain}
  \newtheorem{prop}[thm]{\protect\propositionname}
  \theoremstyle{remark}
  \newtheorem{claim}[thm]{\protect\claimname}
\theoremstyle{plain}
\newtheorem{que}[thm]{\protect\questionname}
\makeatother

  \providecommand{\claimname}{Claim}
  \providecommand{\corollaryname}{Corollary}
  \providecommand{\definitionname}{Definition}
  \providecommand{\examplename}{Example}
  \providecommand{\lemmaname}{Lemma}
  \providecommand{\propositionname}{Proposition}
  \providecommand{\remarkname}{Remark}
\providecommand{\theoremname}{Theorem}
  \providecommand{\questionname}{Question}
  
\newcommand{\con}{\subseteq}

\newcommand{\emp}{\varnothing}
\newcommand{\smin}{\!\smallsetminus\!}

\newcommand{\p}{\mathbb{P}}
\newcommand{\q}{\mathbb{Q}}
\newcommand{\m}{\mathcal{M}}
\newcommand{\fii}{\varphi}
\newcommand{\nec}{\square}
\newcommand{\pos}{\lozenge}

\newcommand{\til}{,\!...,}

\newcommand{\s}{\mathcal{S}}
\newcommand{\frakt}{\mathfrak{t}}
\newcommand{\one}{\mathbbm{1}}
\newcommand{\init}{\trianglelefteq}
\newcommand{\tini}{\trianglerighteq}

\title[Modal logic of $\sigma$-centered forcing]{The modal logic of $\sigma$-centered forcing \\ and related forcing classes}
\author{Ur Ya'ar}

\thanks{This work is an adaptation of the author's MSc thesis \cite{BAT} (under former name), written at the Hebrew University, Jerusalem, under supervision of Prof. Menachem Magidor. I would like to thank Prof. Magidor for the many ideas, discussions and advice, which made this work possible.\\
I would also like to thank the anonymous referee for their helpful comments and feedback, and for asking the questions leading to the discussion in section \ref{overL} and to theorem \ref{thm:allP} }

\begin{document}

\begin{abstract}
	We consider the modality ``$\fii$ is true in every $\sigma$-centered forcing extension'', denoted $\nec\fii$, and its dual ``$\fii$ is true in some $\sigma$-centered forcing extension'', denoted $\pos\fii$ (where $\fii$ is a statement in set theory), which give rise to the notion of a \emph{principle of $\sigma$-centered forcing}. We prove that if ZFC is consistent, then the modal logic of $\sigma$-centered forcing, i.e.\ the ZFC-provable principles of $\sigma$-centered forcing, is exactly $\mathsf{S4.2}$. We also generalize this result to other related classes of forcing.
\end{abstract}
\maketitle

\section{Introduction and preliminaries}

 In this work we continue the investigation of the \emph{Modal Logic of Forcing}, initiated by Joel Hamkins and Benedikt L\"{o}we in \cite{MLF}, where they consider the modal logic arising from considering a statement as necessary (respectively possible) if it is true in any (res. some) forcing extension of the world. Here we restrict the modality only to extensions obtained by $\sigma$-centered focing notions, and prove that the modal logic arising from this interpretation is $\mathsf{S4.2}$ (see below). We then show that our techniques can be generalized to other related classes of forcing notions.

We begin by setting some preliminaries -- first we cite common definitions
and theorems of forcing and of modal logic; and then present the main
tools developed in \cite{MLF,StC} for the research of the modal logic
of forcing; we add one new notion to this set of tools, the notion
of an $n$-switch, and show it's utility; and prove a general theorem (thm. \ref{main-labeling}) which provides the framework for the main theorem (thm. \ref{thm:main}).
 In section \ref{sec:s-centered} we present the class of $\sigma$-centered forcing and some of its properties which
give us the easy part of the theorem -- that the modal logic of $\sigma$-centered
forcing contains $\mathsf{S4.2}$, and present the technique of coding
subsets using $\sigma$-centered forcing. The hard part of the main
theorem will be proved in section \ref{sec:Labeling-frames}, where
we begin by defining a specific model of $\mathrm{ZFC}$, and then
present two forcing constructions that would allow us to establish
that the modal logic of $\sigma$-centered forcing is contained in
$\mathsf{S4.2}$. We conclude with the above-mentioned generalizations
and some open questions.


We begin by presenting some notations and background that will be used
in this work. Our forcing notation is standard, and will usually follow Kunen's \cite[chapter VII]{Kunen}. 

We work with propositional modal logic as presented in \cite{ModalLogic}, in which we add to standard propositional logic two unary operators -- $\nec,\pos$, where $\nec\fii$ is interpreted as ``necessarily $\fii$'' and $\pos\fii$ as ``possibly $\fii$''.

The modal axioms which will be used are: 

\[
\begin{array}{cc}
\mathrm{K} & \nec\left(\fii\to\psi\right)\to\left(\nec\fii\to\nec\psi\right)\\
\mathrm{Dual} & \pos\fii\leftrightarrow\neg\nec\neg\fii\\
\mathrm{T} & \nec\fii\to\fii\\
\mathrm{4} & \nec\fii\to\nec\nec\fii\\
\mathrm{.2} & \pos\nec\fii\to\nec\pos\fii\\
\mathrm{.3}  & \quad\left(\pos\fii\land\pos\psi\right)\to\pos\left[\left(\pos\fii\land\psi\right)\lor \left(\fii\land\pos\psi\right)\right]\\
\mathrm{5} & \pos\nec\fii\to\fii
\end{array}
\]
and the modal theories discussed are: $\mathsf{S4}$, axiomatized by $\mathrm{K, Dual, T}$ and $\mathrm{4}$, $\mathsf{S4.2}$, axiomatized by adding axiom $\mathrm{.2}$, $\mathsf{S4.3}$ by adding axiom $\mathrm{.3}$ and $\mathsf{S5}$ by adding axiom $\mathrm{5}$.

We assume the reader is familiar with \emph{Kripke Semantics} for modal logic, where a \emph{Kripke model} is a triplet $\m =\left<W,R,V\right>$ such that $W$ is a non-empty set (the set of worlds), $R$ is a binary relation on $W$ (the accessibility relation) and $V$ a function from the propositional variables to subsets of $W$ (the valuation);
$\mathcal{F}=\left< W,R \right>$ is called the \emph{frame} on which $\m$ is based; and the satisfaction relation    $\mathcal{M},w\vDash\fii$ (for $w\in W$) is defined in the usual inductive way, using  
\begin{equation*}
\mathcal{M},w\vDash\nec\fii ~\text{iff for every}~ u\in W \text{such that}~ wRu,~ \mathcal{M},u\vDash\fii.
\end{equation*}
We say that $\fii$ is \emph{valid} \emph{in} $\mathcal{M}$ ($\mathcal{M}\vDash\fii$) if $\mathcal{M},w\vDash\fii$ for every $w\in W$, and that that $\fii$ is valid on a frame $\mathcal{F}$ ($\mathcal{F}\vDash\fii$) if $\fii$ is valid in every model based on $\mathcal{F}$. A class of frames $\mathcal{C}$ \emph{characterizes} a modal theory $\Lambda$ if a formula is in $\Lambda$ iff it is valid on every frame in $\mathcal{C}$.

We will use the following class of frames to characterize $\mathsf{S4.2}$:
\begin{defn} \label{def:pBA}
Let $\left\langle F,\leq\right\rangle $ such that $\leq$ is a reflexive
and transitive binary relation on $F$. $\left\langle F,\leq\right\rangle $
is called a \emph{pre-Boolean-algebra }(a pBA) if $\left\langle F/\!\equiv,\leq\right\rangle $ is a Boolean-algebra (a BA), where $\equiv$ is the natural equivalence relation on $F$ defined by $x\equiv y$ iff $x\leq y\leq x$, and $\leq$ denotes also the induced order relation.
\end{defn}
A pBA can be thought of as a BA where every element is replaced by
a cluster of equivalent elements. We will use the following:

\begin{thm}[Thm. 11 in \cite{MLF}]
\label{thm:S4.2 comp}$\mathsf{S4.2}$ is characterized by the class
of all finite pre-Boolean-algebras.
\end{thm}

\subsection{The modal logic of forcing}
We now review the framework of the modal logic of forcing, based on \cite{MLF} and \cite{StC}. The reader who is familiar with these works may wish to skip to definition \ref{def:control-statements} where we define the new notion of an $n$-switch.
 
In the context of set-theory, the possible world semantics suggest a connection between modal logic and forcing, as we can imagine all generic extensions of the universe
(or of a specific model of $\mathrm{ZFC}$) as an enormous Kripke
model (called ``\emph{the generic multiverse}''). This leads naturally to the forcing interpretation of modal logic, in which we say that a sentence of set theory $\fii$ is necessary ($\nec\fii$) if it is true in all forcing extensions, and possible ($\pos\fii$) if it is true in some
forcing extension. 
Given some definable class of forcing notions $\Gamma$, we can also restrict to posets
belonging to that class, to get the operators $\nec_{\Gamma},\pos_{\Gamma}$.
The following definitions, based on \cite{MLF} and \cite{StC}, allow
us to formally ask the question -- what statements are valid under
this interpretation? 
\begin{defn}

\begin{enumerate}
\item Given a formula $\fii=\fii\left(q_{0}\til q_{n}\right)$ in the language
of modal logic, where $q_{0}\til q_{n}$ are the only propositional
variable appearing in $\fii$, and some set-theoretic sentences $\psi_{0}\til\psi_{n}$,
the \emph{substitution instance} $\fii\left(\psi_{0}\til\psi_{n}\right)$
is the set-theoretic statement obtained recursively by replacing $q_{i}$
with $\psi_{i}$ and interpreting the modal operators according to
the forcing interpretation (or the $\Gamma$-forcing interpretation). 
\item Let $\Gamma$ be a class of forcing notions. The $\mathrm{ZFC}$\emph{-provable
principles of~ $\Gamma$-forcing }are all the modal formulas $\fii$
such that $\mathrm{ZFC}\vdash\fii\left(\psi_{0}\til\psi_{n}\right)$
for every substitution $q_{i}\mapsto\psi_{i}$ under the $\Gamma$-forcing
interpretation\emph{. }This will also be called \emph{the modal logic
of~ $\Gamma$-forcing}, denoted $\mathsf{MLF}\left(\Gamma\right)$.
If we discuss the class of all forcing notions we omit mention of
$\Gamma$.
\end{enumerate}
\end{defn}
\begin{thm}[Hamkins and L\"{o}we, \cite{MLF}]
 If ~$\mathrm{ZFC}$ is consistent then the $\mathrm{ZFC}$-provable
principles of forcing are exactly $\mathsf{S4.2}$.
\end{thm}
We will now present the main tools which were developed to prove the
theorem above, and which can be used to prove similar theorems. To
prove such a theorem, we need to establish lower and upper bounds,
i.e. find a modal theory $\Lambda$ such that $\mathsf{MLF}\left(\Gamma\right)\supseteq\Lambda$
and $\mathsf{MLF}\left(\Gamma\right)\con\Lambda$ respectively. Each
type of bound require a different set of tools, which will be presented
below.

\subsubsection{Lower bounds}

A simple observation is that the $\mathrm{ZFC}$-provable principles
of $\Gamma$-forcing are closed under the usual deduction rules for modal logic, 
so if a modal theory is given by some axioms, to show it is contained in $\mathsf{MLF}\left(\Gamma\right)$
it is enough to check that the axioms are valid principles of $\Gamma$-forcing.
So for example, axioms $\mathrm{K}$ and $\mathrm{Dual}$ are easily
seen to be valid under the $\Gamma$-forcing interpretation for every
class $\Gamma$. The validity of other axioms depends on specific
properties of $\Gamma$:
\begin{defn}
A definable class of forcing notions $\Gamma$ is said to be \emph{reflexive}
if it contains the trivial forcing; \emph{transitive} if it is closed
under finite iterations, i.e. if~ $\p\in\Gamma$ and $\dot{\q}$
is a $\p$-name for a poset such that $\Vdash_{\p}\dot{\q}\in\Gamma$,
then $\p*\dot{\q}\in\Gamma$; \emph{persistent} if $\p,\q\in\Gamma$ implies $\q\in\Gamma^{V^\p}$; 
and \emph{directed} if $\p,\q \in\Gamma$ implies that there is some $\mathbb{R}\in\Gamma$ such that $\mathbb{R}$ is forcing equivalent to $\p*\dot{\mathbb{S}}$ and to $\q*\dot{\mathbb{T}}$, where $\dot{\mathbb{S}}\in \Gamma^{V^{\p}} $ and $\dot{\mathbb{T}}\in \Gamma^{V^{\q}}$.

Note that if a $\Gamma$ is transitive and persistent, we can show it is directed by taking $\mathbb{R}=\p\times \q$ for any $\p,\q\in\Gamma$.
\end{defn}

\begin{thm}[Thm. 7 in \cite{StC}] \label{thm:lower bounds}
	Axiom $\mathrm{T}$ is valid in every reflexive forcing class, axiom $\mathrm{4}$ in every transitive forcing class and axiom $\mathrm{.2}$ in every directed forcing class. Thus, if~ $\Gamma$ is reflexive, transitive and directed then $\mathsf{MLF}(\Gamma)\supseteq\mathsf{S4.2}$.
\end{thm}	

\subsubsection{\label{subsec:Upper-bounds}Upper bounds}
To establish that
$\Lambda$ is an upper bound for $\mathsf{MLF}\left(\Gamma\right)$,
we need to show that every formula not in $\Lambda$ is also not in
$\mathsf{MLF}\left(\Gamma\right)$. To do so, we would need to find
a model of $\mathrm{ZFC}$ and some substitution instance of $\fii$
that fails in this model. In the case that $\Lambda$ is 
characterized by some class of frames $\mathcal{C}$, $\fii\notin\Lambda$ means
that there is some Kripke model based on a frame in $\mathcal{C}$
where $\fii$ fails. So our goal would be to find a suitable model of set theory
$W$ such that the $\Gamma$-generic multiverse generated by $W$
(i.e. all $\Gamma$-forcing extensions of $W$) ``looks like'' the
model where $\fii$ fails. The main tool for that is called a labeling:
\begin{defn}
\label{def:labeling}A $\Gamma$\emph{-labeling of a frame $\left\langle F,R\right\rangle $
for a model of set theory $W$} is an assignment to each $w\in F$
a set-theoretic statement $\Phi_{w}$ such that:

\begin{enumerate}
\item The statements form a mutually exclusive partition of truth in the
$\Gamma$-generic multiverse over $W$, i.e. every $\Gamma$-generic
extension of $W$ satisfies exactly one $\Phi_{w}$. 
\item The statements correspond to the relation, i.e. if $W\left[G\right]$
is a $\Gamma$-forcing extension of $W$ that satisfies $\Phi_{w}$,
then $W\left[G\right]\vDash\pos\Phi_{u}$ iff $wRu$.
\item $W\vDash\Phi_{w_{0}}$ where $w_{0}$ is a given initial element of
$F$.
\end{enumerate}
\end{defn}

\begin{lem}[The labeling lemma -- Lem. 9 in \cite{StC}]
\label{lem:labeling}Suppose $w\mapsto\Phi_{w}$ is a $\Gamma$-labeling
of a finite frame $\left\langle F,R\right\rangle $ for a model of
set theory $W$ with $w_{0}$ an initial world of $F$, and $\m$
a Kripke model based on $F$. Then there is an assignment of the propositional
variables $p\mapsto\psi_{p}$ such that for every modal formula $\fii\left(p_{0}\til p_{n}\right)$,
\[
\m,w_{0}\vDash\fii\left(p_{0}\til p_{n}\right)  \iff  W\vDash\fii\left(\psi_{p_{0}}\til\psi_{p_{n}}\right).
\]
\end{lem}

\begin{cor}
\label{cor:S4.2 upper bound} If every finite pre-Boolean-algebra has
a $\Gamma$-labeling over some model of $\mathrm{ZFC}$, then $\mathsf{MLF}\left(\Gamma\right)\con\mathsf{S4.2}$
\end{cor}
\begin{proof}
By theorem \ref{thm:S4.2 comp}, every modal formula $\fii\notin\mathsf{S4.2}$
fails in a Kripke model based on some finite pBA. So, given a $\Gamma$-labeling
for this frame over a model $W$, by the labeling lemma there
is a substitution instance of $\fii$ which fails at $W$ under the
$\Gamma$-forcing interpretation. So $\fii\notin\mathsf{MLF}\left(\Gamma\right)$.
\end{proof}
Hence to establish upper bounds, we try to find labelings for specific
frames. Various labelings can be constructed using certain kinds of
set-theoretic statements, called in general \emph{control statements.}
\begin{defn}[Control statements]\label{def:control-statements}
Let $W$ be some model of set theory, and $\Gamma$ some class of
forcing notions. 

\begin{enumerate}
\item 
	A \emph{switch for $\Gamma$-forcing over $W$ }is a statement $s$
	such that necessarily, both $s$ and $\neg s$ are possible. That is, over every $\Gamma$-extension of $W$ one can force $s$ or $\neg s$ as one chooses using $\Gamma$-forcing.
\item 
	An\emph{ $n$-switch for $\Gamma$-forcing over $W$} is a set of
	statements $\left\{ s_{i}\mid i<n\right\} $ (where $n>1$) such that every $\Gamma$-generic extension $W'$ of $W$ satisfies exactly
	one $s_{i}$, and every $s_j$ is necessarily possible, i.e. over every $\Gamma$-extension of $W$ one can force $s_j$ using $\Gamma$-forcing. The $n$-switch
	value in some $W\left[G\right]$ is the $i$ such that $W\left[G\right]\vDash s_{i}$.
	Note that a $2$-switch is essentially just a switch. 
	\footnote{In \cite{Hamkins-Linnebo}, Hamkins and Linnebo independently define the notion of a "dial", which is essentially the same as an $n$-switch, and prove results similar to our theorem \ref{thm:S4.2 labeling} and lemma \ref{lem:n-switch}. Our results were independent of these.}
\item
 	A \emph{button for $\Gamma$-forcing over $W$} is a statement $b$
	which is necessarily possibly necessary, i.e. $W\vDash\nec\pos\nec b$.
	This means that in every $\Gamma$-extension of $W$, we can force $b$ to be
	true using $\Gamma$-forcing and to remain true in every further $\Gamma$-extension. A button is called
	\emph{pushed} if $\nec b$ holds, otherwise it is called \emph{unpushed}.
	A \emph{pure button} is a button $b$ such that $\nec\left(b\to\nec b\right)$
	(i.e. if it is true then it is pushed). If $b$ is an unpushed button
	then $\nec b$ is an unpushed pure button.
\item 	
	A \emph{ratchet for $\Gamma$-forcing over $W$ }is a 	collection of pure buttons $\left\{ r_{i}\mid i\in I\right\} $, possibly with $i$ as a parameter, where $I$ is well-ordered, such that pushing $r_{i}$
	pushes every $r_{j}$ for $j<i$, and necessarily, every unpushed
	$r_{i}$ can be pushed without pushing any $r_{j}$ for $j>i$. An
	infinite ratchet $\left\{ r_{i}\mid i\in I\right\} $ is called \emph{strong
	}if there is no $\Gamma$-extension of $W$ satisfying every $r_{i}$.
	The ratchet value in $W\left[G\right]$ is the first $i\in I$ such
	that $W\left[G\right]\vDash\neg r_{i}$.
\item 
	A family of control statements (switches, $n$-switches, buttons,
	ratchets) is called \emph{independent over $W$} (for $\Gamma$-forcing) if in $W$, all buttons are
	unpushed (including the ones in any ratchet), and necessarily, using $\Gamma$-forcing, each
	button can be pushed, each switch can be turned on or off, the value
	of each $n$-switch can be changed, and the value of every ratchet
	can be increased, without affecting any other control statement in
	the family.

Note the ``necessarily'' -- the independence needs to be preserved
in any $\Gamma$-forcing extension of $W$.

\end{enumerate}
\end{defn}
$n$-switches are less naturally occurring in set theory than the other notions, and indeed
they were not explicitly defined in \cite{MLF} and \cite{StC}. However,
by examining the proofs of some of the main theorems there, one can
see that what was implicitly used was an $n$-switch, which was constructed
using switches (cf. \cite[theorems 10,11,13]{StC}). Additionally,
in some cases switches were constructed from ratchets and then transformed
into $n$-switches (e.g. in \cite[theorems 12, 15]{StC}). So, in
the definition of some of the central labelings, $n$-switches turn
out to be the more natural notion, and we will show how to construct
them using either switches or a ratchet independently. Hence the following
theorem, which gives sufficient conditions for the existence of labelings
for finite pBA's, generalizes some of the above-mentioned theorems
from \cite{StC}, and they can be inferred from it.
 We will not be able to use the theorem as it is to prove our main theorem, but we will use its proof as a model, so it has instructive value in itself.
\begin{thm}
\label{thm:S4.2 labeling}Let $\Gamma$ be some reflexive and transitive
forcing class and $W$ a model of set theory. If for every $m,n<\omega$ there is a family of $m$ buttons mutually independent from an $n'$-switch for some $n'\geq n$
then there is a $\Gamma$-labeling over $W$ for every frame which is a finite pre-Boolean-algebra.
\end{thm}
\begin{proof}
Let $\left\langle F,\leq\right\rangle $ be a finite pBA. As noted
earlier, it can be viewed a finite BA, where each element is replaced
by a cluster of equivalent worlds. We
can add dummy worlds to each cluster without changing satisfaction in the model, so we can assume that each cluster is of size $n$
for some $1<n<\omega$. \footnote{If every cluster has only one element then we actually don't need the $n$-switch, and we can label the BA only with the buttons.}
 It is known that any finite BA is isomorphic
to the BA $\left\langle \mathcal{P}\left(B\right),\con\right\rangle $
for some finite set $B$. Let $B$ be such that $\left\langle F/\equiv,\leq\right\rangle \cong\left\langle \mathcal{P}\left(B\right),\con\right\rangle $,
and set $m=\left|B\right|$. We can assume that in fact $B=\left\{ 0\til m-1\right\} $.
There is a correspondence between subsets $A\con B$ and clusters
in $\left\langle F,\leq\right\rangle $. Each cluster is of size $n$,
so by enumerating each cluster, all the elements of $F$ can be named
$w_{i}^{A}$ for $i<n$ and $A\con B$, where $w_{i}^{A}\leq w_{j}^{A'}$
iff $A\con A'$. An initial world in $F$ must be in the bottom cluster,
which corresponds to $\emp\con B$ so without loss of generality it is enumerated as $w_{0}^{\emp}$.

By the assumption, adding more dummy worlds to each cluster if needed and increasing $n$,
there are buttons $\left\{ b_{0}\til b_{m-1}\right\} $ and an $n$-switch
$\left\{ s_{0}\til s_{n-1}\right\} $ all independent of each other
over $W$. We can assume the buttons are pure. To define a labeling,
each cluster, corresponding to some $A\con B$, will be labeled by
the statement that the only buttons pushed are the ones with indexes
from $A$. Inside each cluster, each world will be labeled by the
corresponding value of the $n$-switch. Formally, we set 
\[
\Phi\left(w_{i}^{A}\right)=\bigwedge_{j\in A}b_{j}\land\bigwedge_{j\notin A}\neg b_{j}\land s_{i}
\]
 and claim that this is a labeling as required by verifying the conditions:

\begin{enumerate}
\item If $W\left[G\right]$ is a $\Gamma$-generic extension of $W$, define
$A=\left\{ j<m\mid W\left[G\right]\vDash b_{j}\right\} $. By the
definition of the $n$-switch, $W\left[G\right]\vDash s_{i}$ for
some unique $i<n$. So it is clear that $W\left[G\right]\vDash\Phi\left(w_{i}^{A}\right)$,
and that for any other pair $\left(A',i'\right)\ne\left(A,i\right)$
with $A'\con B$, and $i'<n$, $W\left[G\right]\nvDash\Phi\left(w_{i'}^{A'}\right)$.
So these statements indeed form a mutually exclusive partition of
truth in the $\Gamma$-generic multiverse over $W$.
\item Assume $W\left[G\right]$ is a $\Gamma$-generic extension of $W$
such that $W\left[G\right]\vDash\Phi\left(w_{i}^{A}\right)$. 

If $w_{i}^{A}\leq u$, then as we have seen, $u=w_{i'}^{A'}$ for
some $i'<n$ and $A\con A'\con B$. By the assumption of independence
of the control statements, we can, by $\Gamma$-forcing, push all
the buttons in $A'\smin A$ (and only them) and change the $n$-switch
value to $i'$ (if needed), to obtain an extension of $W\left[G\right]$
satisfying $\Phi\left(w_{i'}^{A'}\right)$. Note that by the transitivity
of $\Gamma$, the $b_{i}$'s are still independent pure buttons in
$W\left[G\right]$, since every $\Gamma$-extension of $W\left[G\right]$
is also a $\Gamma$-extension of $W$. In particular, any button true
in $W\left[G\right]$ remains true in the extension. So $W\left[G\right]\vDash\pos\Phi\left(w_{i'}^{A'}\right)$
as required.

If $W\left[G\right]\vDash\pos\Phi\left(w_{i'}^{A'}\right)$, then
there is some extension $W\left[G\right]\left[H\right]\vDash\Phi\left(w_{i'}^{A'}\right)$.
By the definition of pure buttons and the reflexivity of $\Gamma$,
$W\left[G\right]\vDash\bigwedge_{j\in A}b_{j}$ implies $W\left[G\right]\vDash\bigwedge_{j\in A}\nec b_{j}$,
so $W\left[G\right]\left[H\right]\vDash\bigwedge_{j\in A}b_{j}$.
Therefore by the definition of $\Phi\left(w_{i'}^{A'}\right)$, we
must have $A\con A'$, so $w_{i}^{A}\leq w_{i'}^{A'}$.
\item A part of the definition of independence is that no button is pushed
in $W$ (since they are pure and $\Gamma$ reflexive, it is equivalent
to saying none is true). We can assume without loss of generality that $W\vDash s_{0}$.
So $W\vDash\Phi\left(w_{0}^{\emp}\right)$.\qedhere
\end{enumerate}
\end{proof}

\begin{cor}
\label{cor:S4.2 contains}Under the assumptions of theorem \ref{thm:S4.2 labeling},
$\mathsf{MLF}\left(\Gamma\right)\con\mathsf{S4.2}$.
\end{cor}
\begin{proof}
Apply corollary \ref{cor:S4.2 upper bound}.
\end{proof}
\begin{lem}
\label{lem:n-switch}An $n$-switch can be produced using the following
control statements:

\begin{enumerate}
\item Independent switches $s_{0}\til s_{m-1}$ if $n=2^{m}$;
\item A strong ratchet $\left\{ r_{i}\mid i\in I\right\} $ where $I$ is
either a limit ordinal or $Ord$, the class of all ordinals, and $i\in I$
is a parameter in $r_{i}$.
\item A family of independent buttons $\left\langle b_{i}\mid i\in I\right\rangle $
where $I$ is as above, with no extensions where all of the buttons are pushed.
\end{enumerate}
\end{lem}
\begin{proof}
For (1), if $j<2^{m}$ let $\bar{s}_{j}$ be the statement that the
pattern of switches corresponds to the binary digits of $j$, that
is, 
\[
\bigwedge\left\{ s_{i}\mid\mbox{the \ensuremath{i}-th binary digit of \ensuremath{j}\,\ is \ensuremath{1}}\right\} \land\bigwedge\left\{ \neg s_{i}\mid\mbox{the \ensuremath{i}-th binary digit of \ensuremath{j}\,\ is \ensuremath{0}}\right\} .
\]
 Clearly in any extension exactly one pattern of the switches holds,
so exactly one $\bar{s}_{j}$ holds. By the independence of the switches,
any pattern can be forced over any extension.

For (2), every $i\in I$ is an ordinal, so of the form $\omega\cdot\alpha+k$
for some $\alpha\in Ord$ and $k<\omega$. Then we let $\bar{s}_{j}$
be the statement ``if $i=\omega\cdot\alpha+k$ is the first such
that $\neg r_{i}$ then $k\bmod n=j$''. Since no extension satisfies
all the $r_{i}$'s, there is always some $i$ which is the first such
that $\neg r_{i}$, and therefore there is some unique $j$ such that
$\bar{s}_{j}$ holds. Since it is a ratchet, in every extension, for
every $j'<n$, we can increase its value to some $i'=\omega\cdot\alpha'+k'$
for some $k'>k$ such that $k'\bmod n=j$ (we use the assumption that
if $I$ is an ordinal then it is a limit).

(3) is similar to (2) by setting $r_i=(\forall j<i)\,b_{j}\land\neg b_{i}$
\end{proof}

So with our previous theorem, we get the following:
\begin{cor}[{\cite[theorems 13 and 15]{StC}
		\footnote{In \cite[theorem 15]{StC} the authors have a slightly different convention, where the ratchet value is the last button which \emph{is} pushed, and they use the notion of a \emph{uniform ratchet}, but the theorem is essentially the same.}}]
\label{cor: S4.2 contains 2} Let $\Gamma$ be some reflexive and
transitive forcing class and $W$ a model of set theory. If there
are arbitrarily large finite families of buttons mutually independent
with arbitrarily large finite families of switches, with a strong
ratchet as above or with another family of independent buttons as
above, then there is a $\Gamma$-labeling for every frame which is
a finite pre-Boolean-algebra over $W$. So in such cases, $\mathsf{MLF}\left(\Gamma\right)\con\mathsf{S4.2}$.
\end{cor}

Note that in the above corollary, 
we must have independent buttons which are not all pushed. However, the following theorem shows that we can weaken this assumption, given that we have \emph{some} $n$-switch -- not necessarily independent. It is this modification which will eventually be used to find a $\sigma$-centered labeling for finite pBA's.

\begin{thm} \label{main-labeling}
	 Let $\Gamma$ be some reflexive and transitive forcing class and $W$ a model of set theory. If there is a family $\left\langle b_{i}\mid i \in \omega\right\rangle $ of independent buttons for $\Gamma$-forcing over $W$, where $i$
	is a parameter in $b_{i}$, and for every $n$ there is an $n$-switch  for $\Gamma$-forcing over $W$, then there is a $\Gamma$-labeling for every finite pBA, and thus $\mathsf{MLF}\left(\Gamma\right)\con\mathsf{S4.2}$.
\end{thm}
\begin{proof}	
	Let $\left\langle F,\leq\right\rangle $ be a pBA. As before  Let $B=\{0,...,m-1\}$ be such that $\left\langle F/\!\equiv,\leq\right\rangle $
	is isomorphic to $\left\langle \mathcal{P}\left(B\right),\con\right\rangle$, and every cluster is without loss of generalization of size $n$ or some $1<n<\omega$.
	So every world in $F$ is of the form $w_{j}^{C}$ for some $j<n$
	and $C\con B$, and we have $w_{j}^{C}\leq w_{j'}^{C'}$ iff $C\con C'$.
	
	Let $\left\langle b_{i}\mid i \in \omega\right\rangle $ be as in the assumption, and using $i\mapsto i-m+1$ we rename it as $\left\langle b_{i}\mid m-1\leq i < \omega\right\rangle $. We will imitate the proof of theorem \ref{thm:S4.2 labeling} by using the statements $b_{i}$ for $m-1\leq i\leq0$ as the buttons, and obtaining from $\left\langle b_{i}\mid 0<i<\omega\right\rangle $ an ``almost'' $n$-switch as in \ref{lem:n-switch}. It might not be a real $n$-switch, if there is some $\Gamma$-extension in which unboundedly many buttons are pushed, and for that reason we need the additional $n$-switch from the assumption. 
	
	So let $\left\{ s_{j}\mid j<n\right\}$ be the $n$-switch from the assumption. To define the ``almost'' $n$-switch, define the following statements: 
	\begin{align*}
	R_0 &= \text{``}\neg b_i \text{ holds for every } 0<i<\omega\text{''} \\
	\nonumber R_{j}&= \text{``\ensuremath{j} is the largest such that \ensuremath{b_j} holds'' (for \ensuremath{j<\omega})} \\
	\nonumber R_{\omega}&= \text{``}\sup\left\{ n\mid b_{n} \text{ holds}\right\} =\omega \text{''}
	\end{align*}

	So $R_{0}$ holds iff no button is pushed, and if in some $\Gamma$-extension of $W$ we have $R_{j}$ for $0<j<\omega$, then in particular we have $b_{j}\land \neg b_{l}$ for any $l>j$. So by the independence of the buttons, if some extension satisfies $R_i$ ($i<\omega$), we can force with some $\Gamma$-forcing to push only $b_l$ for any $l>i$ and obtain exactly $R_{l}$. Note that if some $R_{j}$ for $j<\omega$ holds, it means in particular that the number of pushed buttons is	bounded. Now, for every $j<n$ we define the statement:
	\begin{align*}	
	 t_{j}&=\text{``There is some } k<\omega \text{ such that } k\equiv j \bmod n \text{ and } R_k \text{ holds''.}
	\end{align*}
	So in any $\Gamma$-extension of $W$, if $ t_{j}$
	holds for some $j$, there is some $k$ be such that $R_{k}$ holds, and for every $j'<n$ we can find $k'>k$ with $k'\equiv j'\bmod n$ and then force to push only $b_k'$ to obtain $R_{k'}$ and thus $ t_{j'}$. It is also clear that no two distinct $ t_{j}$'s can hold at the same time, and that if the number of $i$'s such that $b_i$ holds is bounded,
	then some $ t_{j}$ holds. So, $\left\{  t_{j}\mid j<n\right\} $
	functions as an $n$-switch, but only as long as the number of pushed
	buttons is bounded. If in some $\Gamma$-extension there are unboundedly
	many buttons pushed (which we allow as a possibility), no $R_{k}$
	holds, so also no $ t_{j}$ holds. Hence this is ``almost'' an $n$-switch. 
	
	Now we are ready to define the labeling. For every $C\con B$, define 
	\[
	\Psi_{C}=\bigwedge_{i\in C} b_{-i}\land\bigwedge_{i\notin C}\neg b_{-i}
	\]
	which states that the pushed buttons out of $\left\{ b_{-i}\mid i\in B\right\} $
	are exactly the ones corresponding to the elements in $C$. These statements label the cluster we are in. To move within each cluster below the topmost one, we will use the ``almost''	$n$-switch $\left\{  t_{j}\mid j<n\right\} $, and if we can no longer
	use it, that is, if there are unboundedly many $b_{i}$'s pushed,
	we put ourselves in the top cluster, and there we move using the $n$-switch
	$\left\{  s_{j}\mid j<n\right\} $: for every $C\subseteq B$ and $j<n$ set
	\begin{equation}
		\Phi\left(w_{j}^{C}\right)= \begin{cases}
		 	\Psi_{C}\land t_{j} & C\ne B\\
			\left(\Psi_{C}\land t_{j}\right)\lor\left(R_{\omega}\land s_{j}\right) & C=B
	\end{cases}
		\end{equation}
	 In this way, the fact that $\left\{  s_{j}\mid j<n\right\} $ is not independent of the buttons will not affect us,	as we will always stay in the top cluster anyway. We will now show
	that this is indeed a labeling as required.

	\subsubsection*{The statements are mutually exclusive:}
	
	It is clear that the statements $\left\{ \Psi_{C}\mid C\con B\right\} $
	are mutually exclusive, so $\Phi\left(w_{j}^{C}\right),\Phi\left(w_{j'}^{C'}\right)$
	for $C\ne C'$, both different than $B$, clearly exclude each other.
	If we look at $\Phi\left(w_{j}^{C}\right)$ and $\Phi\left(w_{j'}^{B}\right)$
	for some $C\ne B$, they exclude each other since if $\Phi\left(w_{j'}^{B}\right)$
	holds, then either $\Psi_{B}$ holds which excludes $\Psi_{C}$, or we have
	$\sup\left\{ n\mid b_{n}\text{ holds}\right\} =\omega$,
	which excludes $ t_{j}$. Now for $j\ne j'$ if $C\subsetneq B$,
	$\Phi\left(w_{j}^{C}\right)$ and $\Phi\left(w_{j'}^{C}\right)$ exclude
	each other since $ t_{j}$ and $ t_{j'}$ exclude each other;
	and if $C\!=\!B$, if $\sup\{ n\mid b_{n}\text{ holds}\} =\omega$
	then $ s_{j}$ and $ s_{j'}$ exclude each other, and otherwise
	again $ t_{j}$ and $ t_{j'}$ exclude each other.

	\subsubsection*{The statements exhaust the truth over $\Gamma$-extensions of $W$:}
	
	Let $W\left[G\right]$ be some $\Gamma$-extension
	of $W$. If $W\left[G\right]\vDash\sup\left\{ n\mid b_{n}\text{ holds}\right\} =\omega$,
	then there is some $j$ such that $W\left[G\right]\vDash s_{j}$,
	and so $W\left[G\right]\vDash\Phi\left(w_{j}^{B}\right)$. Otherwise,
	the number of buttons pushed is finite, so there is some $j$ such
	that $W\left[G\right]\vDash t_{j}$, and there is also some specific
	subset of the buttons $\left\{ b_{-i}\mid i\in B\right\} $ which
	are pushed in $W\left[G\right]$, so there is some $C\con B$ such
	that $W\left[G\right]\vDash\Psi_{C}$, and together we get $W\left[G\right]\vDash\Phi\left(w_{j}^{C}\right)$.
	
	\subsubsection*{$W$ satisfies $\Phi\left(w_{0}^{\protect\emp}\right)$:}
	
	In $W$ we have $\neg b_{i}$ for all $m-1\leq i<\omega$, so $W\vDash\Psi_{\emp}$ and also $W\vDash R_{0}$, and therefore	also $W\vDash s_{0}$.
	
	\subsubsection*{The statements correspond to the relation:}
	
	Assume we are in $U$ which is a $\Gamma$-extension
	of $W$ where $\Phi\left(w_{j}^{C}\right)$ is true.
	
	Assume first that $C\ne B$.
	\begin{itemize}
		\item Assume $\pos\Phi\left(w_{j'}^{C'}\right)$ -- there is a
		$\Gamma$-extension $U'$ of $U$ satisfying $\Phi\left(w_{j'}^{C'}\right)$.
		If $C'\ne B$ then 
		\[
		U'\vDash\Psi_{C'}=\bigwedge_{i\in C'}b_{-i}\land\bigwedge_{i\notin C'}\neg b_{-i}.
		\]
		But $U\vDash\bigwedge_{i\in C}b_{-i}$, which are buttons, so
		they remain pushed in $U'$, i.e. $U'\vDash\bigwedge_{i\in C}b_{-i}$.
		So we must get $C\con C'$, so $w_{j}^{C}\leq w_{j'}^{C'}$. If $C'=B$
		then clearly we have $w_{j}^{C}\leq w_{j'}^{C'}$.
		\item Assume $w_{j}^{C}\leq w_{j'}^{C'}$, hence $C\con C'$. We have 
		\[
		U\vDash\Psi_{C}=\bigwedge_{i\in C}b_{-i}\land\bigwedge_{i\notin C}\neg b_{-i},
		\]
		so for every $i\in C'\smin C$, by the independence of the buttons we can force $b_{-i}$, to obtain	an extension $U'$ satisfying $\Psi_{C'}$ (the buttons from $C$	will remain pushed). In $U$, which satisfies $ t_{j}$, there
		is some $k$ such that $k\bmod n=j$ and $U\vDash R_{k}$. In $U'$
		we still have $R_{k}$, since pushing the (finitely many) buttons corresponding to $C'$ does not push any button $b_{i}$ for $i>0$. If $j'=j$ we are done, otherwise we can find some $k'>k$, such that $k'\bmod n=j'$, push $b_{k'}$
		and thus obtain an extension $U''$ satisfying $ t_{j'}$. Again
		this forcing does not affect the truth of $b_{-i}$'s for $i\in B$, so $U''$ also satisfies $\Psi_{C'}$, so it satisfies $\Phi\left(w_{j'}^{C'}\right)$. By
		transitivity of $\Gamma$, we get that indeed $U\vDash\pos\Phi\left(w_{j'}^{C'}\right)$.
	\end{itemize}
	Now assume $C=B$, i.e. $U\vDash\Phi\left(w_{j}^{B}\right)$. We distinguish
	the two cases. 
	\begin{itemize}
		\item $U\vDash\left(\Psi_{B}\land t_{j}\right)$:
		
		\begin{itemize}
			\item Assume $\pos\Phi\left(w_{j'}^{C'}\right)$. Since $U\vDash\Psi_{B}$,
			any extension of it also satisfies $\Psi_{B}$, so we cannot have
			$\pos\Phi\left(w_{j'}^{C'}\right)$ for any $C'\ne B$. Therefore
			$C'=B$ and indeed $w_{j}^{B}\leq w_{j'}^{B}=w_{j'}^{C'}$.
			\item Assume $w_{j}^{C}\leq w_{j'}^{C'}$. So $C'=B$ as well. $U\vDash t_{j}$,
			so as we have seen before, we can force over $U$ to obtain a generic
			extension satisfying $ t_{j'}$. This extension will still satisfy
			$\Psi_{B}$ since these are buttons, so it will satisfy $\Phi\left(w_{j'}^{B}\right)$ as required. 
		\end{itemize}
		\item $U\vDash\left(\sup\{ n\mid b_{n} \text{ holds}\} =\omega\right)\land s_{j}$:
		
		\begin{itemize}
			\item Assume $\pos\Phi\left(w_{j'}^{C'}\right)$. Since $U\vDash\sup\left\{ n\mid b_{n}\text{ holds}\right\} =\omega$, any extension of $U$ also satisfies this, since these are buttons, so we cannot have $\pos\Phi\left(w_{j'}^{C'}\right)$
			for any $C'\ne B$. Therefore $C'=B$ and indeed $w_{j}^{B}\leq w_{j'}^{C'}$.
			\item Assume $w_{j}^{C}\leq w_{j'}^{C'}$. So $C'=B$ as well. $U\vDash s_{j}$,
			so by the definition of an $n$-switch, we can force over $U$ to obtain a generic	extension satisfying $ s_{j'}$. This extension will still satisfy
			$\sup\left\{ n\mid b_{n} \text{ holds}\right\} =\omega$ since
			these are buttons, so it will satisfy $\Phi\left(w_{j'}^{B}\right)$
			as required. 
		\end{itemize}
	\end{itemize}
	Hence we have defined a $\Gamma$ labeling for the frame	$\left\langle F,\leq\right\rangle $ over $W$. This was for every pBA, so by corollary \ref{cor:S4.2 upper bound},  $\mathsf{MLF}\left(\Gamma\right)\con\mathsf{S4.2}$.
\end{proof}

We end this section by citing another theorem of this sort, which we will use in section \ref{related classes}:
\begin{thm}[{\cite[theorem 12]{StC}}]\label{thm:ratchet->S4.3}If there is a long ratchet over a model of set theory $ W $, i.e. a strong ratchet $ \langle r(\alpha) | 0<\alpha \in \mathrm{Ord}\rangle $, where $r(\alpha)$ is obtained by a single formula with parameter $\alpha$, then $ \mathsf{MLF}(\Gamma)\con\mathsf{S4.3} $.
\end{thm}

\section{\label{sec:s-centered}$\sigma$-centered forcing }
We now proceed to the investigation of the modal logic of a specific
class of forcing notions -- the class of all $\sigma$-centered forcing
notions.
\begin{defn}
Let $\p$ be any poset.

\begin{enumerate}
\item A subset $C\con\p$ is called \emph{centered} if any finite number
of elements in $C$ have a common extension in $\p$.
\item A poset is called $\sigma$\emph{-centered} if it is the union of
countably many centered subsets.
\end{enumerate}
\end{defn}
\begin{rem}
\label{rem:For-convenience}For convenience we will always assume
that the top element $\one_{\p}$ is in each of the centered posets.
This does not affect the generality since every element is compatible
with it. It will also sometimes be convenient to assume that if $\p=\bigcup_{n\in\omega}\p_{n}$
where each $\p_{n}$ is centered, then each $\p_{n}$ is upward closed,
i.e. if $q\in\p_{n}$ and $q\leq p$ then $p\in\p_{n}$. This also
doesn't affect the generality since if $q_{1}\til q_{k}\in\p_{n}$
and $q_{i}\leq p_{i}$ then a common extension for the $q_{i}$'s
will also extend the $p_{i}$'s . 
\end{rem}
The following is a central example for a $\sigma$-centered forcing,
versions of which will be used later on:
\begin{defn} \label{def:ad-forcing}
Let $Y$ be a subset of $\mathcal{P}\left(\omega\right)$. We define
a poset $\p_{Y}$ as follows:

\begin{itemize}
\item The elements are of the form $\left\langle s,t\right\rangle $ where
$s$ is a finite subset of $\omega$ and $t$ a finite subset of $Y$;
\item $\left\langle s,t\right\rangle $ is extended by $\left\langle s',t'\right\rangle $
if $s\con s'$, $t\con t'$ and for every $A\in t$, $s\cap A=s'\cap A$.
\end{itemize}
\end{defn}
So we think of the first component as finite approximations for a
generic real $x\con\omega$, while the second component limits our
options in extending the approximation. A condition $p=\left\langle s,t\right\rangle $
tells us that $s\con x$ and that for every $A\in t$, $x\cap A=s\cap A$,
so that the intersection of $x$ with any set in $Y$ will turn out
to be finite.
\begin{lem}
For any $Y\con\mathcal{P}\left(\omega\right)$, $\p_{Y}$ is $\sigma$-centered. 
\end{lem}
\begin{proof}
Note that if $t_{1}\til t_{n}$ are finite subsets of $Y$, then for
any $s\in\left[\omega\right]^{<\omega}$, the conditions $\left\langle s,t_{1}\right\rangle \til\left\langle s,t_{n}\right\rangle $
are all extended by $\left\langle s,t_{1}\cup\dots\cup t_{n}\right\rangle $.
 So $\p$ is the union of the centered posets
$\p_{s}=\left\{ \left\langle s,t\right\rangle \mid t\con Y\,\mathrm{finite}\right\} $.
Since there are only countably many finite subsets of $\omega$, we
get that $\p$ is $\sigma$-centered. 
\end{proof}
We will explore the properties of this kind of posets in section \ref{sec:Almost-disjoint-forcing}. The following lemma lists a few well-known properties of $\sigma$-centered forcing.

\begin{lem}\label{s.c.properties}
	\begin{enumerate}
		\item Every $\sigma$-centered poset has the c.c.c. and thus preserves cardinals and cofinalities.	
	
		\item \label{lem:continuum} Assume $\lambda\geq\aleph_{0}$, $2^{\lambda}=\kappa$ and let $\p$ be some $\sigma$-centered forcing notion. Then $V^{\p}\vDash2^{\lambda}=\kappa$. 
		
		\item \label{lem:Productivity}Let $\left\langle \p_{\alpha}\mid\alpha<\lambda\right\rangle $
		for some $\lambda<\left(2^{\aleph_{0}}\right)^{+}$ be a collection
		of $\sigma$-centered posets. Let $\p=\prod_{\alpha<\lambda}\p_{\alpha}$
		be the finite support product \footnote{All products in this paper are of finite support} of $\left\langle \p_{\alpha}\mid\alpha<\lambda\right\rangle $.
		Then $\p$ is also $\sigma$-centered.
		
		\item \label{lem:transitivity}If~ $\p$ is a $\sigma$-centered posets
			and~ $\dot{\q}$ is a $\p$-name such that $\p$ forces that $\dot{\q}$
			is a $\sigma$-centered posets, then also $\p*\dot{\q}$ is $\sigma$-centered.
	\end{enumerate}
\end{lem}

Note that (\ref{lem:transitivity}) shows that $\Gamma_{\sigma\text{-centered}}$,
the class of all $\sigma$-centered forcing notions, is transitive. It is also reflexive since the trivial forcing is trivially $\sigma$-centered, and persistent since being the union of countably many centered subsets is an upward absolute notion. So, using theorem \ref{thm:lower bounds},
we have the following:
\begin{thm}
\label{thm:contained in}The $\mathrm{ZFC}$-provable principles of
$\sigma$-centered forcing contain $\mathsf{S4.2}$.
\end{thm}
Finally, it will be of use to know that essentially, $\sigma$-centered
forcing notions are ``small'', so there aren't too many of them:
\begin{lem}
\label{lem:small}Let $\p$ be a $\sigma$-centered forcing notion.
Then the separative quotient of~ $\p$ is of size at most $2^{\aleph_{0}}$.
\end{lem}
\begin{proof}
Let $\p=\bigcup_{n\in\omega}\p_{n}$ where each $\p_{n}$
is upward closed. Recall that the separative quotient is the quotient of $\p$ by the equivalence relation: $x,y\in\p$, $x\sim y$
iff $\left\{ z\in\p\mid z\parallel x\right\} =\left\{ z\in\p\mid z\parallel y\right\} $.
We denote the equivalence class of $x$ by $\left[x\right]$. Define
for every $x\in\p$ 
\[
A\left(x\right)=\left\{ n\in\omega\mid x\in\p_{n}\right\} .
\]
 We claim that for every $x,y\in\p$, $\left[x\right]\ne\left[y\right]$
implies $A\left(x\right)\ne A\left(y\right)$. $\left[x\right]\ne\left[y\right]$
means that (without loss of generality) there is some $z\parallel x$ such that $z\bot y$.
Let $z'$ be a common extension of $z$ and $x$. So $z'\bot y$ as
well (otherwise $z$ would be compatible with $y$). Let $n\in\omega$
such that $z'\in\p_{n}$. Since we assumed $\p_{n}$ is upward closed,
also $x\in\p_{n}$, so $n\in A\left(x\right)$. Assume towards contradiction
that $n\in A\left(y\right)$, i.e. $y\in\p_{n}$. But $\p_{n}$ is
centered, so $y$ and $z'$ must be compatible, which is a contradiction.

So we get that 
$
\left|\p/\!\sim\right|\leq\left|\left\{ A\left(x\right)\mid x\in\p\right\} \right|\leq\left|\mathcal{P}\left(\omega\right)\right|=2^{\aleph_{0}}
$.
\end{proof}
\begin{cor}
\label{cor:few forcings}Up to forcing-equivalence, there are at most
$2^{2^{\aleph_{0}}}$ $\sigma$-centered forcing notions. 
\end{cor}
\begin{proof}
Every poset is forcing equivalent to it's separative quotient

 and by the previous lemma there are at most $2^{2^{\aleph_{0}}}$ of those.
\end{proof}

\subsection{\label{sec:Almost-disjoint-forcing}Almost disjoint forcing}

In this section we present one of the tools for labeling frames
with $\sigma$-centered forcing -- almost disjoint forcing, which is
a version of the example introduced in the previous section. The results
in this section are due to Jensen and Solovay in \cite{JS}.

Two infinite sets are called \emph{almost disjoint} (a.d.) if their intersection
is finite. We would like to have a way to construct almost disjoint subsets of
$\omega$ in a very definable and absolute way. For that, we fix some
recursive enumeration $\frakt=\left\langle \frakt_{i}\mid i<\omega\right\rangle $ of all
finite sequences of $\omega$ (which we will also use later on, note that $\frakt\in L$ and is absolute), and define for every $f:\omega\to\omega$
\[
\mathcal{S}\left(f\right)=\left\{ i<\omega\mid \frakt_{i}\mathrm{\:is\:an\:initial\:segment\:of\:}f\right\} .
\]
If $f,g$ are distinct then $\mathcal{S}\left(f\right)$
and $\mathcal{S}\left(g\right)$ are almost disjoint.
Hence, $\left\{ \s\left(f\right)\mid f:\omega\to\omega\right\} $ is
a family of $2^{\aleph_{0}}$ pairwise a.d. subsets of~ $\omega$.

From the discussion at section 2.4 of \cite{JS} we have the following:
\begin{thm}
	Let $\mathcal{F}\in M$ be a family of a.d.
	subsets of~ $\omega$, $Y\con\mathcal{F}$ (in $M$), and $\p_{Y}$ the forcing from definition \ref{def:ad-forcing}. Then forcing with $\p_{Y}$
	adds a real $x$ such that for every $y\in\mathcal{F}$, $x\cap y$
	is finite iff~ $y\in Y$.
\end{thm}

So $\p_{Y}$ adds a generic real $x$ which is almost-disjoint from each member of $Y$. Furthermore, if $x$ is obtained by the generic filter $G$, then clearly $M[G]=M[x]$.
This gives us a method to code subsets of $2^{\omega}$ using subsets
of $\omega$. Let $M$ be some model of $\mathrm{ZFC}$, set 
 an enumeration $\left\{ f_{\alpha}\mid\alpha<\kappa\right\} \in M$
of $\omega^{\omega}$ (where $\kappa=\left(2^{\omega}\right)^{M}$),
and define as before $\mathcal{F}=\left\{\mathcal{S}\left(f_{\alpha}\right)\mid\alpha<\kappa\right\} $ which are a.d.
So for each $A\con\kappa$, $A\in M$, we can define $Y=Y\left(A\right)=\left\{ \mathcal{S}\left(f_{\alpha}\right)\mid\alpha\in A\right\} $,
and force with $\p_{Y}$ to obtain a generic real $x=x_{A}$, and by
the previous theorem, $\alpha\in A$ iff $\mathcal{S}\left(f_{\alpha}\right)\in Y$
iff $x\cap\mathcal{S}\left(f_{\alpha}\right)$ is finite. So, in $M\left[x\right]$,
we get that 
\[
A=\left\{ \alpha<\kappa\mid\mathcal{S}\left(f_{\alpha}\right)\cap x\;\mathrm{is\,finite}\right\} 
\]
(note that $\p_{Y}$ preserves both cardinals and the continuum, so
$\kappa=\left(2^{\omega}\right)^{M\left[x\right]}$ , and if $\kappa=\aleph_{\alpha}^{M}$
for some $\alpha$, then also $\kappa=\aleph_{\alpha}^{M\left[x\right]}$). In this case, we say that ``$x$ codes $A$''. 

\section{\label{sec:Labeling-frames} Control statements for $\sigma$-centered forcing}

Our goal in this section is to prove that the modal logic of
$\sigma$-centered forcing is contained in $\mathsf{S4.2}$ by producing control statements for $\sigma$-centered forcing that will meet the requirements of theorem \ref{main-labeling}. We begin by describing a specific model $W$ which will be our ground
model, and then construct the independent family of buttons and the $n$-switches required in the theorem.

\subsection{\label{subsec:ground-model}The ground model}

We begin with the constructible universe $L$, and use Cohen forcing to obtain mutually generic reals
$\left\langle a_{\alpha,i}\mid\alpha<\omega_{1}^{L},i<\omega\right\rangle $
i.e. each $a_{\beta,j}$ is generic over $L\left[\left\langle a_{\alpha,i}\mid\alpha<\omega_{1}^{L},i<\omega,\left(\alpha,i\right)\ne\left(\beta,j\right)\right\rangle \right]$. Let 
\[Z=L\left[\left\langle a_{\alpha,i}\mid\alpha<\omega_{1}^{L},i<\omega\right\rangle \right].\]
Our ground model $ W $ is a generic extension of $ Z $, which preserves the mutual genericity of $ \left\langle a_{\alpha,i}\mid\alpha<\omega_{1}^{L},i<\omega\right\rangle $, such that these reals are ordinal-definable with a definition which is absolute for generic extensions of $ W $ by $ \sigma $-centered forcing. This can be done e.g. by using Easton forcing to code the reals in the power function above some large enough cardinal. Any extension of $ W $ for which the above definition is absolute will be called an \emph{appropriate extension}.
We will also require that in $ W $ we do not collapse cardinals and add no new subsets below $ \aleph_\omega $, so e.g. $\omega_{1}^{L}=\omega_{1}^{W}$. From now on we'll
deal with forcings which do not collapse cardinals (by c.c.c), and
also do not change the continuum (by lemma \ref{s.c.properties}(\ref{lem:continuum})),
so we omit such superscripts.

\subsection{\label{subsec:Buttons}The buttons}

Now over $W$ we can define $T_{i}$, for $i<\omega$ as the statement:
\begin{gather*} \label{T-predicates}
\text{For every real $x$ and for all but boundedly many } \alpha<\omega_{1}, \\
a_{\alpha,i}\text{ is Cohen generic over } L\left[x,\left\langle a_{\beta,j}\mid\beta<\omega_{1},j\ne i\right\rangle \right]. \nonumber
\end{gather*}
Since the reals $a_{\alpha,i}$ and the sequences $\left\langle a_{\beta,j}\mid\beta<\omega_{1},j\ne i\right\rangle $
are ordinal definable in $W$ and it's appropriate extensions, also $L\left[x,\left\langle a_{\beta,j}\mid\beta<\omega_{1},j\ne i\right\rangle \right]$
is definable with $x$ as a parameter. So, formally, $T_{i}$
includes the definitions of these elements, which will be interpreted
as we expect in all relevant models. The question whether a real $r$
is generic over some definable submodel is also expressible in the
language of set theory, as it just means that $r$ is in every open
dense subset of $\omega^{\omega}$ which is in that model. So $T_{i}$
is indeed a sentence in the language of set theory. Note that if we
want, by slight abuse of notation we can treat $i$ as a ``variable''
denoting a natural number, rather than a definable term; thus we would
be able to phrase sentences such as $\forall i<\omega\,T_{i}$. This
will be used in the next section. In this section when we talk about
a specific $T_{i}$, we take $i$ to be a fixed term.
\begin{rem}

\begin{enumerate}
\item $W\vDash T_{i}$ for every $i$: We required that we do not add any new subsets of $\omega^{\omega}$ or any new real.
So every real $x\in W$ is already in $Z$.
Fix some $i<\omega$ and a real $x\in W$. This real was introduced
by at most boundedly many $a_{\alpha,i}$, that is, there is some
$\gamma<\omega_{1}$ such that 
\[
x\in L\left[\left\langle a_{\alpha,j}\mid\alpha<\omega_{1},j\ne i\right\rangle \cup\left\langle a_{\alpha,i}\mid\alpha<\gamma\right\rangle \right].
\]
 All the reals $a_{\alpha,i}$ for $\alpha>\gamma$ are generic over
the above model so also above $L\left[x,\left\langle a_{\beta,j}\mid\beta<\omega_{1},j\ne i\right\rangle \right]$
. 
\item $\neg T_{i}$ is a pure button for appropriate extensions: if for
some $a_{\alpha,i}$ there is some real $x$ such that $a_{\alpha,i}$
is not generic over $L\left[x,\left\langle a_{\beta,j}\mid\beta<\omega_{1},j\ne i\right\rangle \right]$,
then it will never again be generic over this model. So, if we destroy
$T_{i}$, we can never get it back as long as it keeps it's above
meaning. Note that if an extension is not appropriate, then $T_{i}$
might have a completely different meaning than what is intended, as
the definitions we use will give some different sets, so it is paramount
we stick with appropriate extensions.
\end{enumerate}
\end{rem}
We will now define forcing notions which will allow us to destroy
$T_{i}$, by destroying the genericity of the relevant $a_{\alpha,i}$'s.
\begin{defn}
In $W$, we define $\p_{i}$ to be the forcing notion with conditions
of the form $\left\{ U_{s_{1}}\til U_{s_{n}},a_{\alpha_{1},i}\til a_{\alpha_{l},i}\right\} $
where $n,l<\omega$, $s_k\in \omega^{<\omega}$ and $U_{s_{k}}\con\omega^{\omega}$ is the basic open set $\{x\in\omega^{\omega}\mid s_k\init x\}$;
and for conditions $p,q\in \p_i$, $q\leq p$ iff $p\con q$ and whenever $a_{\alpha,i}\in p$ and $U_s\in q\smin p$, $a_{\alpha,i}\notin U_s$.
That is, to extend a
condition, we can add any finite number of the reals, and we can add
any finite number of basic open sets, as long as the new sets do not
include any of the old reals. 
\end{defn}
We will show that the forcing $\p_{i}$ destroys the genericity of
all the $a_{\alpha,i}$'s, by adding dense open sets (approximated
by the $U_{s}$'s) that do not include them. So, intuitively, a condition
$p=\left\{ U_{s_{1}}\til U_{s_{n}},a_{\alpha_{1},i}\til a_{\alpha_{l},i}\right\} $
states which reals will be avoided in subsequent stages.
\begin{rem}
\label{rem:extension} Given some distinct $a_{\alpha_{1},i}\til a_{\alpha_{l},i}$ and $s\in \omega^{<\omega}$
we can always find some $s'\tini s$ such that $a_{\alpha_{1},i}\til a_{\alpha_{l},i}\notin U_{s'}$:
let $t=\bigcap_{k=1}^{l}a_{\alpha_{k},i}$, i.e. the longest initial
segment common to $a_{\alpha_{1},i}\til a_{\alpha_{l},i}$. If $s\init t$ or $t\init s$, let $t'$ be the longer of the two and assume it is of length $n$. Take some $j\in\omega\smin\left\{ a_{\alpha_{1},i}\left(n\right)\til a_{\alpha_{l},i}\left(n\right)\right\} $ and set $s'=t'^{\frown}\left\langle j\right\rangle $, then $a_{\alpha_{1},i}\til a_{\alpha_{l},i}\notin U_{s'}$ and $s'\tini s$. Otherwise $s'=s$ will do. 
\end{rem}
Let $G\con\p_{i}$ be a generic filter. Note that by the former remark,
the set of conditions having at least $n$ basic-open sets in them
is dense in $\p_{i}$ (given a condition $p$, we can find an $s$
such that $U_{s}$ does not contain any of the reals in $p$, and
then add to $p$, e.g., $U_{s},U_{s^{\frown}\left\langle 0\right\rangle },U_{s^{\frown}\left\langle 0,0\right\rangle }...$
to obtain an extension with at least $n$ basic-open sets). So, the
conditions in $G$ give us an infinite sequence $\left\langle U_{s_{k}}\mid k<\omega\right\rangle $ of basic-open sets. We assume that $\left\langle s_{k}\mid k<\omega\right\rangle $ forms a subsequence of the recursive enumeration $\left\langle \frakt_{i}\mid i<\omega\right\rangle$, i.e. there are indexes $\left\langle n_{k}\mid k<\omega\right\rangle $ such that $s_{k}=\frakt_{n_k}$ for every $k$.
\begin{lem}
\label{lem:dense}For every $k<\omega$, the set  $\bigcup_{n\geq k}U_{s_{n}}$
is open-dense in $\omega^{\omega}$.
\end{lem}
\begin{proof}
It is clearly open as a union of open sets. To show it is dense, let $s\in\omega^{<\omega}$ and  we need to find some $n\geq k$ such that $s_n\tini s$.
Note that as in remark \ref{rem:extension}, for every $p\in \p_i $ we can extend $s$ to some $s'\tini s$ such that $p\cup \{U_t\}\leq p$, and we can also make sure that $|s'|>N$ for any fixed $N$. So by genericity there is some $p\in G$ containing some $U_{s'}$ where $s'\tini s$ and $|s'|>\max \{|s_l| \mid l<k\}$ so in particular $s'=s_n$ for $n\geq k$ as required.
\end{proof}
\begin{lem}
For every $\alpha<\omega_{1}$ there is some $k$ such that $a_{\alpha,i}\notin\bigcup_{n\geq k}U_{s_{n}}$.
\end{lem}
\begin{proof}
Fix $\alpha<\omega_{1}$ and let $D_{\alpha}=\left\{ p\in\p_{i}\mid a_{\alpha,i}\in p\right\} $.
So $D_{\alpha}$ is clearly open, and it is dense since for every
$p$, $p\cup\left\{ a_{\alpha,i}\right\} $ is a legitimate extension
of $p$ (we did not limit the addition of $a_{\beta,i}$'s). So there
is some $p\in G\cap D_{\alpha}$. Let $k$ be larger than any $n$
such that $U_{s_{n}}\in p$. We want to show that $a_{\alpha,i}\notin\bigcup_{n\geq k}U_{s_{n}}$.
Otherwise, there is some $n\geq k$ such that $a_{\alpha,i}\in U_{s_{n}}$.
So there is some $q\in G$ with $U_{s_{n}}\in q$, and by moving to
a common extension we can assume $q\leq p$.
In fact, $q<p$, since $U_{s_{n}}\notin p$ by the choice of $k$
and $n$. But $a_{\alpha,i}\in p$, $q<p$ and $U_{s_{n}}\in q\smin p$
imply that $a_{\alpha,i}\notin U_{s_{n}}$, by contradiction.
\end{proof}
So indeed, $\p_{i}$ adds open-dense sets which destroy the genericity
of every $a_{\alpha,i}$. This will show that $T_{i}$ is destroyed,
once we show that $T_{i}$ still means the same thing after forcing
with $\p_{i}$.
\begin{lem}
$\p_{i}$ is $\sigma$-centered.
\end{lem}
\begin{proof}
For every $t_{1}\til t_{n}\in\omega^{<\omega}$, let $\p\left(t_{1}\til t_{n}\right)$
be the set of all conditions in $\p_{i}$ containing exactly the basic-open
sets $U_{t_{1}}\til U_{t_{n}}$. Note that there are only $\omega$
such sets $\left\{ t_{1}\til t_{n}\right\} $, and that clearly $\p_{i}=\bigcup\{\p\left(t_{1}\til t_{n}\right)\mid t_{1}\til t_{n}\in\omega^{<\omega}\}$.
Now notice that every $\p\left(t_{1}\til t_{n}\right)$ is centered,
since if $p_{1}\til p_{l}\in\p\left(t_{1}\til t_{n}\right)$, then
$p_{1}\cup\dots\cup p_{l}$ is still a legitimate condition in $\p_{i}$,
and it extends each $p_{j}$ since the only limitation on extension
concerned the basic-open sets, which we did not change. 
\end{proof}
\begin{cor}
Let $W'$ be some appropriate extension of~ $W$. Let $G\con\p_{i}$
be generic over $W'$. Then $W'\left[G\right]\vDash\neg T_{i}$
\end{cor}
\begin{proof}
By $\sigma$-centeredness, after forcing with $\p_{i}$
the meaning of all the definitions in $T_{i}$ remain the same. So
we will find a real $x\in W'\left[G\right]$ such that all the $a_{\alpha,i}$'s
are already not generic over $L\left[x\right]$, so surely $T_{i}$
fails. 
Recall the enumeration $ \langle \frakt_n \mid n<\omega\rangle$ we fixed earlier, and define $x=\left\{ m\mid\exists p\in G\left(U_{\frakt_{m}}\in p\right)\right\} $.
So, if as before $\left\langle U_{s_{n}}\mid n<\omega\right\rangle $
is the sequence of basic-open sets given by $G$, we assumed it is a subsequence of  $ \langle \frakt_n \mid n<\omega\rangle$, so in $L\left[x\right]$ we can already define each
union $\bigcup_{n\geq k}U_{s_{n}}$. Hence, as we have shown above,
each $a_{\alpha,i}$ is not in some dense-open set of $\omega^{\omega}$
in $L\left[x\right]$, and therefore not generic over $L\left[x\right]$
as required.
\end{proof}
Our next task will be to show that forcing with some $\p_{j}$ does
not affect the truth of $T_{i}$ for any $i\ne j$.
\begin{lem}
Let $W'$ be an appropriate extension of~ $W$, such that $W'\vDash T_{i}$.
Let $G\con\p_{j}$ be generic over $W'$ for $j\ne i$. Then $W'\left[G\right]\vDash T_{i}$. 
\end{lem}
\begin{proof}
Assume otherwise. Then there is some $x\in W'\left[G\right]$ such
that unboudedly many $a_{\alpha,i}$ are not generic over $L\left[x,\left\langle a_{\beta,k}\mid\beta<\omega_{1},k\ne i\right\rangle \right]$.
Let $\dot{x}\in W'$ be a $\p_{j}$-name for $x$. Since $x$ is a
real, we can assume that $\dot{x}$ is a name containing only elements
of the form $\left\langle q,\check{n}\right\rangle $ for $n\in\omega$
and $q\in\p_{j}$. Furthermore, since $\p_{j}$ is c.c.c, we can assume
that there are only countably many elements of the form $\left\langle q,\check{n}\right\rangle $
for each $n$.
 So $\dot{x}$ is a countable collection of elements of the form $\left\langle q,\check{n}\right\rangle $. We wish to ``code'' $\dot{x}$ by some real $y\in W'$. We do this
in the usual way: Let $\gamma$ be the supremum of all $\alpha<\omega_{1}$ such that
$a_{\alpha,j}\in q$ for some $\left\langle q,\check{n}\right\rangle \in\dot{x}$.
Since $\dot{x}$ is countable and each such $q$ contains only finitely
many $a_{\alpha,j}$'s, $\gamma<\omega_{1}$.  Each $q\in\p_{j}$ is of the form $\left\{ U_{s_{1}}\til U_{s_{n}},a_{\alpha_{1},j}\til a_{\alpha_{l},j}\right\} $,
so it is determined by a finite subset of $\omega^{<\omega}$ and
a finite subset of ordinals no larger than $\gamma$. This information can be coded by a finite sequence of natural numbers $z_{q}$. Each pair $\left\langle z_{q},\check{n}\right\rangle $ can be coded
by a natural number.  So the entire $\dot{x}$ can be coded by a set of natural numbers
$y$. All these codings are done in $W'$ so $y\in W'$.

Now assume that $W'[G]\vDash$``$ a_{\alpha,i}$ is not generic over $M':=L\left[x,\left\langle a_{\beta,k}\mid\beta<\omega_{1},k\ne i\right\rangle \right]$''.
We'll show that already $W'\vDash$``$ a_{\alpha,i}$ is not generic over $M:=L\left[y,\left\langle a_{\beta,k}\mid\beta<\omega_{1},k\ne i\right\rangle \right]$''. 

$M\con W'$, and since the definition of $\p_{j}$ requires only the
reals $\left\langle a_{\beta,j}\mid\beta<\omega_{1}\right\rangle $,
$\p_{j}\in M$. In addition, since we can decode $y$ in this model,
we have $\dot{x}\in M$. Since $\p_{j}\in M\con W'$, $G$ is generic
also over $M$. The fact that $a_{\alpha,i}$ is not generic over
$M'$ means that there is a dense open set $U\in M'$ such that  in $W'[G]$, $a_{\alpha,i}\notin U$. From the perspective of $M$ and $M'$, $a_{\alpha,i}$ is merely a Cohen generic over  $L\left[\left\langle a_{\beta,k}\mid\beta<\omega_{1},k\ne i\right\rangle \right]$.
Since $\dot{x}\in M$, $M'\con M\left[G\right]$, so $U\in M\left[G\right]\con W'[G]$.
So there is some $p\in G$ and some $\p_{j}$-name $\dot{U}\in M$
such that $p\Vdash$ ``$\dot{U}$ is an open-dense subset of $\omega^{\omega}$, which does not contain reals which are Cohen generics over $L\left[\left\langle a_{\beta,k}\mid\beta<\omega_{1},k\ne i\right\rangle \right]$''. Define
\[
\bar{U}=\left\{ r\in\omega^{\omega}\mid\exists p'\leq p\left(p'\Vdash\check{r}\in\dot{U}\right)\right\} .
\]
So $\bar{U}\in M$. We claim first that $\bar{U}$ is open-dense. 

Open: let $r\in\bar{U}$, witnessed by $p'\leq p$ s.t. $p'\Vdash\check{r}\in\dot{U}$.
Since $p'$ also forces that $\dot{U}$ is open, there is some $p''\leq p'$
and some $s\in\omega^{<\omega}$ such that $p''\Vdash``\check{r}\in U_{s}\con\dot{U}"$,
that is $p''\Vdash\check{s}\init\check{r}\land\left(\check{s}\init\dot{r}'\to\dot{r}'\in\dot{U}\right)$.
The initial segment relation does not change, so $s\init r$. If $r'\in U_{s}$,
then in particular we'll get $p''\Vdash\check{r}'\in\dot{U}$, so
by definition $r'\in\bar{U}$. So $U_{s}\con\bar{U}$.
So $\bar{U}$ is open. 

Dense: Let $s\in\omega^{<\omega}$. Since $p$ forces that $\dot{U}$
is open-dense, there is some $p'\leq p$ and some $t\tini s$ such
that $p'\Vdash\left(\check{t}\init\dot{r}\to\dot{r}\in\dot{U}\right)$.
So let some $t\init r\in\omega^{\omega}$, then in
particular we get $p'\Vdash\check{r}\in\dot{U}$ , so $r\in\bar{U}$
and $r\tini s$ as required.

Second, we claim that in $W'$, $a_{\alpha,i}\notin\bar{U}$. Otherwise, there
is some $p'\leq p$ such that $p'\Vdash\check{a}_{\alpha,i}\in\dot{U}$.
But $p$ forced that no Cohen generic over $L\left[\left\langle a_{\beta,k}\mid\beta<\omega_{1},k\ne i\right\rangle \right]$ is in $\dot{U}$,
a contradiction.

So, we have found an open-dense set $\bar{U}\in M$ such that $a_{\alpha,i}\notin\bar{U}$,
so $a_{\alpha,i}$ is not generic over $M$. This was for every $a_{\alpha,i}$
not generic over $L\left[x,\left\langle a_{\beta,k}\mid\beta<\omega_{1},k\ne i\right\rangle \right]$,
and we assumed there are unboundedly many of these. So there are unboundedly
many $a_{\alpha,i}$'s which are not generic over $M=L\left[y,\left\langle a_{\beta,k}\mid\beta<\omega_{1},k\ne i\right\rangle \right]$,
where $y\in W'$. But this contradicts the assumption that $W'\vDash T_{i}$.
So, we indeed get that also $W'\left[G\right]\vDash T_{i}$.
\end{proof}
To conclude, packing up what we have done in this section, we obtain
the following:
\begin{prop} \label{prop:buttons}
$\left\{ \neg T_{i}\mid i<\omega\right\} $ is a family of independent
buttons over $W$ for $\sigma$-centered forcing.
\end{prop}
\begin{rem}
In fact, we can replace ``$\sigma$-centered'' with any every reflexive
and transitive class of forcing notions, containing all the $\p_{i}$'s,
such that every extension of $W$ with a forcing from the class yields
an appropriate extension.

Note that if it were the case that in no extension of $W$ by $\sigma$-centered
forcing all these buttons are pushed, we could have finished the proof
of our main theorem using corollary \ref{cor: S4.2 contains 2}. However,
by lemma \ref{s.c.properties}(\ref{lem:Productivity}), $\prod_{i<\omega}\p_{i}$ is $\sigma$-centered,
and it pushes all the buttons, so we will have to use theorem \ref{main-labeling}.
\end{rem}

\subsection{\label{subsec:n-switches}The $n$-switches}

\begin{prop} \label{prop:n-switch}
Let $M$ be a model of $\mathrm{ZFC}$ such that $M\vDash2^{\aleph_{0}}=\aleph_{1}\land 2^{\aleph_{1}}=\aleph_{2}$ and every subset of $\omega_{2}$ is ordinal-definable using a definition which is absolute to $\sigma$-centered forcing extensions. Then for every $n>1$ there is an $n$-switch for $\sigma$-centered forcing over $M$.
\end{prop}
\begin{proof}
Let $\left\langle f_{\alpha}\mid\alpha<\omega_{1}\right\rangle \in M$
be a definable enumeration of all the functions $f:\omega\to\omega$
in $M$ 
 and define $y_{\alpha}=\mathcal{S}\left(f_{\alpha}\right)$ as in section \ref{sec:Almost-disjoint-forcing}.
Let $\left\langle A_{\xi}\mid\xi<\omega_{2}\right\rangle \in M$ be
some fixed definable enumeration of all the subsets 
of $\omega_{1}$ in $M$. Let $C\left(x,\xi\right)$
be the statement 
\begin{gather}
x\con\omega\text{ and }(\alpha\in A_{\xi}\leftrightarrow x\cap y_{\alpha}\,\mathrm{is\,finite})
\end{gather}
referred to as ``$x$ is a real coding $A_{\xi}$''. By the discussion
at the end of section \ref{sec:Almost-disjoint-forcing}, for every
$\xi$ there is a $\sigma$-centered forcing notion $\q_{\xi}\in M$
such that $\q_{\xi}\Vdash\exists xC\left(x,\xi\right)$. We would
like to define a ratchet by using the
statements ``$\alpha=\sup\left\{ \xi<\omega_{2}\mid\exists xC\left(x,\xi\right)\right\} $''.
By defining so, we can indeed always increase the value of alleged ratchet
by forcing with $\q_{\alpha}$, but in a certain extension,
forcing with $\q_{\alpha}$ might also add a real coding $A_{\xi}$
some $\xi>\alpha$. To fix that, we will define an unbounded set $\mathcal{E}$
such that adding a code for $A_{\alpha}$ for some $\alpha\in\mathcal{E}$
doesn't add a code for any larger $A_{\xi}$.

We work now within $M$, and fix some $\sigma$-centered poset $Q\in M$.
Let $\alpha<\omega_{2}$. We define by induction $\left\{ \alpha_{\zeta}\mid\zeta<\omega_{1}\right\} $.
Set $\alpha_{0}=\alpha$. If $\alpha_{\zeta}<\omega_{2}$ is defined
for $\zeta<\omega_{1}$, let 
\[
\alpha_{\zeta+1}=\sup\Big\{ \beta<\omega_{2}\mid Q\times\prod\limits _{\xi\leq\alpha_{\zeta}}\q_{\xi}\nVdash\neg\exists xC\left(x,\check{\beta}\right)\Big\} +1.
\]
The above set is not empty since $\q_{\alpha_{\zeta}}\Vdash\exists xC\left(x,\check{\alpha}_{\zeta}\right)$.
In particular, $\alpha_{\zeta}<\alpha_{\zeta+1}$. 

\begin{claim}
$\alpha_{\zeta+1}<\omega_{2}$.
\end{claim}
\begin{proof}
Let $P=Q\times\prod_{\xi\leq\alpha_{\zeta}}\q_{\xi}$. Since $\alpha_{\zeta}<\omega_{2}=\left(2^{\aleph_{0}}\right)^{+}$,
this is $\sigma$-centered by lemma \ref{s.c.properties}(\ref{lem:Productivity}) so by lemma \ref{s.c.properties} and our assumptions on $M$, $P\Vdash2^{\aleph_{1}}=\aleph_{2}>2^{\aleph_{0}}$.
In particular, $P\Vdash``\sup\left\{ \beta\mid\exists xC\left(x,\beta\right)\right\} <\omega_{2}"$,
since there cannot be $\aleph_{2}$ reals each coding a different
subset of $\omega_{1}$. Note that by the c.c.c of $P$ there can only be $\aleph_{0}$
many possible values for $\sup\left\{ \beta\mid\exists xC\left(x,\beta\right)\right\} $.
Hence 
 there is some $\beta_{P}<\omega_{2}$
bounding all these possible values, so it is forced by $P$
that $\sup\left\{ \beta\mid\exists xC\left(x,\beta\right)\right\} \leq\beta_{P}$.
Now, if for some $\gamma$, $P\nVdash\neg\exists xC\left(x,\check{\gamma}\right)$,
then there is $p\in P$ such that $p\Vdash\exists xC\left(x,\check{\gamma}\right)$,
but also $p\Vdash``\sup\left\{ \beta\mid\exists xC\left(x,\beta\right)\right\} \leq\beta_{P}"$ so 
$\gamma\leq\beta_{P}$.
So by the definition, $\alpha_{\zeta+1}\leq\beta_{P}+1<\omega_{2}$. 
\end{proof}
For $\zeta<\omega_{1}$ limit, set $\alpha_{\zeta}=\sup\left\{ \alpha_{\xi}\mid\xi<\zeta\right\} $
($\omega_{2}$ is regular so also in this case $\alpha_{\zeta}<\omega_{2}$),
and finally let $\alpha^{*}=\sup\left\{ \alpha_{\zeta}\mid\zeta<\omega_{1}\right\} $.
Again 
$\alpha^{*}<\omega_{2}$. 

\begin{claim}
\label{claim:sup}Let $G$ be generic for $\q=Q\times\prod\limits _{\xi<\alpha^{*}}\q_{\xi}$
such that $M\left[G\right]\vDash\exists xC\left(x,\beta\right)$.
Then $\beta<\alpha^{*}$. 
\end{claim}
\begin{proof}
Let $x\in M\left[G\right]$ such that $M\left[G\right]\vDash C\left(x,\beta\right)$.
So there is a $\q$-name $\tau$ and some $p_{*}\in G$ which forces
$C\left(\tau,\check{\beta}\right)$. For every $n$, let $C_{n}\con\q$
be a maximal antichain below $p_{*}$ of conditions deciding the statement
$\check{n}\in\tau$. By the c.c.c each $C_{n}$ is countable, so also
$C=\bigcup_{n\in\omega}C_{n}$ is countable. Every element of $\q$
is of the form $\langle q,\left(p_{\gamma}\right)_{\gamma<\alpha^{*}}\rangle $
where only for finitely many $\gamma$'s $p_{\gamma}\ne\one_{\q_{\gamma}}$.
So for each $p\in C$, denote this finite set of ordinals by $F_{p}$,
and let $\gamma^{*}=\sup\bigcup_{p\in C}F_{p}$. Each $F_{p}$ is
a set of ordinals less than $\alpha^{*}$, so $\gamma^{*}\leq\alpha^{*}$.
But since $C$ is countable and each $F_{p}$ is finite, $\gamma^{*}$
has at most countable cofinality, while $\alpha^{*}$ is the limit
of an increasing $\omega_{1}$ sequence $\left\langle \alpha_{\zeta}\mid\zeta<\omega_{1}\right\rangle $,
so $\gamma^{*}<\alpha^{*}$, and furthermore, there is some $\zeta<\omega_{1}$
such that $\gamma^{*}\leq\alpha_{\zeta}$. Let $\bar{\q}=Q\times\prod_{\xi\leq\alpha_{\zeta}}\q_{\xi}$.
For every $p\in\q$, if $p=\langle q,\left(p_{\gamma}\right)_{\gamma<\alpha^{*}}\rangle$,
let $\bar{p}=\langle q,\left(p_{\gamma}\right)_{\gamma\leq\alpha_{\zeta}}\rangle $,
and let $\bar{G}=\left\{ \bar{p}\mid p\in G\right\} $.
$\bar{G}$ is $\bar{\q}$-generic over $M$. We claim that
$x\in M\left[\bar{G}\right]$ and $M\left[\bar{G}\right]\vDash C\left(x,\beta\right)$.
Define the $\bar{\q}$-name 
\[
\bar{\sigma}=\bigcup_{n\in\omega}\big(\left\{ \bar{p}\mid p\in C_{n},p\Vdash\check{n}\in\tau\right\} \times\left\{ \check{n}\right\}\big) .
\]

\begin{rem*}
By the choice of $\alpha_{\zeta}\geq\gamma^{*}$, for every $p\in C_{n}$,
if it is of the form $\langle q,\left(p_{\gamma}\right)_{\gamma<\alpha^{*}}\rangle $,
then $p_{\gamma}=\one$ for every $\gamma>\alpha_{\zeta}$.
\end{rem*}
If $n\in x$ then there is some $p\in G$, $p\leq p_{*}$ that forces
$\check{n}\in\tau$. By the maximality of $C_{n}$, it intersects
$G$, which is a filter, so a condition in the intersection must also
force $\check{n}\in\tau$ (and not $\check{n}\notin\tau$). So we
can choose such $p\in C_{n}\cap G$, and by definition $\left\langle \bar{p},\check{n}\right\rangle \in\bar{\sigma}$.
In addition, $\bar{p}\in\bar{G}$, so $n\in\bar{\sigma}_{\bar{G}}$. 

If $n\in\bar{\sigma}_{\bar{G}}$ there is some $p\in C_{n}$, $p\Vdash\check{n}\in\tau$
such that $\bar{p}\in\bar{G}$. So, there is some $r\in G$ such that
$\bar{r}=\bar{p}$. Note that by the remark $r$ and $p$ are equal
in every coordinate where $p$ is not trivial, so $r\leq p$. Therefore
also $r\Vdash\check{n}\in\tau$, and $r\in G$ so $n\in x$.

So, we get that $\bar{\sigma}_{\bar{G}}=x$, so $x\in M\left[\bar{G}\right]$.
$M\left[\bar{G}\right]\con M\left[G\right]$, and in $M\left[G\right]$
we have $C\left(x,\beta\right)$, i.e. 
\[
A_{\beta}=\left\{ \alpha\mid x\cap y_{\alpha}\,\mathrm{is\,finite}\right\} ,
\]
so we can already have this equation in $M\left[\bar{G}\right]$ (since
the $y_{\alpha}$'s don't change), so $M\left[\bar{G}\right]\vDash C\left(x,\beta\right)$.
Therefore, we have that $Q\times\prod _{\xi\leq\alpha_{\zeta}}\q_{\xi}\nVdash\neg\exists xC\left(x,\beta\right)$,
so by the definition of $\alpha_{\zeta+1}$, $\beta\leq\alpha_{\zeta+1}<\alpha^{*}$,
as required.
\end{proof}
So we have defined an operation $\alpha\mapsto\alpha^{*}$ for every
$\alpha<\omega_{2}$ (note that this operation was relative to $Q$).
Since for every $\alpha<\omega_{2}$ we have $\alpha<\alpha^{*}<\omega_{2}$,
the set $\left\{ \alpha^{*}\mid\alpha<\omega_{2}\right\} $ is unbounded
in $\omega_{2}$, so the set $\mathcal{C}_{Q}$ consisting of all
limit points of this set is a club.

By corollary \ref{cor:few forcings}, in $M$ there are at most $(2^{2^{\aleph_{0}}})^{M}$
$\sigma$-centered forcing notions up to equivalence, or to be exact,
at most $(2^{2^{\aleph_{0}}})^{M}$ separative $\sigma$-centered forcing
notions. By our assumptions on $M$, this cardinal is $\aleph_{2}$. Note
that $\left\{ \prod_{\alpha<\xi}\q_{\alpha}\mid\xi\in\mathcal{C}_{Q}\right\} $ where $Q$ is the trivial forcing 
are non-equivalent $\sigma$-centered posets, so we get exactly $\aleph_{2}$
posets. Since every separative $\sigma$-centered poset has size at
most $\aleph_{1}$, each can be coded as binary relations on $\omega_{1}$,
so by our assumption we can have a defineable enumeration
$\left\langle Q_{\zeta}\mid\zeta<\omega_{2}\right\rangle$ of all the separative $\sigma$-centered
forcing notions in $M$. Define $\mathcal{C}$ as be the diagonal
intersection of $\left\langle \mathcal{C}_{Q_{\zeta}}\mid\zeta<\omega_{2}\right\rangle $:
\[
\mathcal{C}:=\underset{{\scriptstyle \zeta<\omega_{2}}}{\triangle}\mathcal{C}_{Q_{\zeta}}=\Big\{ \alpha<\omega_{2}\mid\alpha\in\bigcap_{\zeta<\alpha}\mathcal{C}_{Q_{\zeta}}\Big\} 
\]
which is also a club in $\omega_{2}$. 
Now we let 
\[
\mathcal{E}=\mathcal{C}\cap\left\{ \alpha<\omega_{2}\mid\mathrm{cf}\left(\alpha\right)=\omega_{1}\right\} .
\]
$\mathcal{E}$ is unbounded: it is known that the set $\left\{ \alpha<\omega_{2}\mid\mathrm{cf}\left(\alpha\right)=\omega_{1}\right\} $
is stationary in $\omega_{2}$ 
so it intersects the club $\left\{\alpha\in\mathcal{C}\mid\alpha>\gamma\right\} $
for each $\gamma<\omega_{2}$. Let $\left\langle e_{\alpha}\mid\alpha<\omega_{2}\right\rangle$
be a (definable) ascending enumeration of $\mathcal{E}$. 

\begin{rem*}
We would have preferred to work with $\mathcal{C}$ rather than $\mathcal{E}$.
The problem is that the $\phantom{}^{*}$ operation may not be continuous at limits of countable cofinality -- to prove continuity, we would like to imitate the proof of claim \ref{claim:sup}, but it requires that the limit is of uncountable cofinality. If the length of the product is of countable cofinality, there might be a real that is not introduced in any bounded product.
\end{rem*}
Now we can define the following statements 
 \begin{equation*}\label{def:ratchet}
	r\left(\alpha\right)=``\alpha=\min\left\{ \beta<\omega_{2}\mid\neg\exists xC\left(x,e_{\beta}\right)\right\}"
\end{equation*}
These are indeed statements (with ordinal parameters perhaps) invoking the definition of $\left\langle e_{\alpha}\mid\alpha<\omega_{2}\right\rangle$,
which retains its intended meaning in every
extension of $M$ by $\sigma$-centered forcing.
 Now given $n>1$, define for every $j<n$ 
 \begin{equation*}
 \Phi_{j}=``r\left(\omega\cdot\alpha+k\right)\to\left(k\bmod n=j\right)".
 \end{equation*}
 
We claim that $\left\{ \Phi_{j}\mid j<n\right\} $ is an $n$-switch
for $\sigma$-centered forcing over $W$.


Let $Q\in M$ be some $\sigma$-centered poset, and $G\con Q$ generic
over $M$. By lemma \ref{s.c.properties}(\ref{lem:continuum}) $M\left[G\right]$ satisfies
$2^{\aleph_{0}}=\aleph_{1}$ and $2^{\aleph_{1}}=\aleph_{2}$ , so in particular the set $\left\{ \beta<\omega_{2}\mid\exists x C\left(x,e_{\beta}\right)\right\} ^{M\left[G\right]}$
is bounded, since there are only $\omega_{1}$ reals, and therefore
there cannot be unboundedly many subsets of $\omega_{1}$ coded by
them. So there is some unique $\gamma<\omega_{2}$ such that $M\left[G\right]\vDash r\left(\gamma\right)$.
There are some unique $j,k<n$ such that $\gamma=\omega\cdot\alpha+k$
and $k\bmod n=j$, so $M\left[G\right]\vDash\Phi_{j}$. Hence every
$\sigma$-centered extension of $M$ satisfies exactly one $\Phi_{j}$.
Now we need to show that for every $j'\ne j$ there is some $\sigma$-centered
extension of $M\left[G\right]$ satisfying $\Phi_{j'}$. Recall the
club $\mathcal{C}_{Q}$ from the above construction. $Q=Q_{\xi}$
for some $\xi<\omega_{2}$. By the unboundedness of $\mathcal{E}$,
we can find some $\gamma'$ such that $e_{\gamma'}>\xi$ and $\gamma'=\omega\cdot\alpha+j'$
for some $\alpha$. We want to show that we have a generic extension
of $M\left[G\right]$ satisfying $r\left(\gamma'\right)$.

Let $H\con\prod_{\zeta<e_{\gamma'}}\q_{\zeta}$ be generic over $M\left[G\right]$.
By the product lemma, $M\left[G\times H\right]=M\left[G\right]\left[H\right]$
and $G\times H$ is $Q\times\prod_{\zeta<e_{\gamma'}}\q_{\zeta}$
generic over $M$. So for every $\beta<e_{\gamma'}$, $\prod_{\zeta<e_{\gamma'}}\q_{\zeta}\Vdash\exists xC\left(x,\beta\right)$,
so $M\left[G\right]\left[H\right]\vDash\exists xC\left(x,\beta\right)$.
On the other hand, recall that $e_{\gamma'}$ is in the diagonal intersection of the
clubs $\mathcal{C}_{Q_{\zeta}}$, so by definition, and since $e_{\gamma'}>\xi$,
$e_{\gamma'}\in\bigcap_{\zeta<e_{\gamma'}}\mathcal{C}_{Q_{\zeta}}\con\mathcal{C}_{Q_{\xi}}$.
So, by the definition of $\mathcal{C}_{Q_{\xi}}$, $e_{\gamma'}$
is either of the form $\delta^{*}$ for some $\delta$ and the $\phantom{}^{*}$
operation corresponding to $Q_{\xi}$, or a limit point of such points.
In the first case, we can just apply claim \ref{claim:sup}. In the
second case, since $e_{\gamma'}\in\mathcal{E}$ is of uncountable
cofinality, we can repeat the proof of claim \ref{claim:sup} with
a sequence $\langle \delta_{\zeta}^{*}\mid\zeta<\omega_{1} \rangle $
that witnesses $e_{\gamma'}\in\mathcal{C}_{Q_{\xi}}$, and get that
the statement in claim \ref{claim:sup} is true as well. That is,
in both cases, we get that if $M\left[G\times H\right]\vDash\exists xC\left(x,\beta\right)$ then $\beta<e_{\gamma'}$. So $e_{\gamma'}$ is the first $\beta$
such that $M\left[G\times H\right]\vDash\neg\exists xC\left(x,\beta\right)$.
Since the enumeration of $\mathcal{E}$ is increasing, we get that
$\gamma'=\min\left\{ \beta<\omega_{2}\mid\neg\exists xC\left(x,e_{\beta}\right)\right\} $.
So $M\left[G\times H\right]\vDash r\left(\gamma'\right)$, and since
$\gamma'=\omega\cdot\alpha+j'$, $M\left[G\times H\right]\vDash\Phi_{j'}$
as required.
\end{proof}

The forcing notions used in this $n$-switch add real numbers in a
rather uncontrollable way, so it is indeed likely that they might
add some real which destroys the genericity of the $a_{\alpha,i}$'s,
therefore it is unlikely that this $n$-switch is independent of the
buttons $\neg T_{i}$. However, by using both constructions presented
in this section, we can overcome the drawbacks each of them has, and obtain our main result:

\begin{thm} \label{thm:main}
If ~$\mathrm{ZFC}$ is consistent, then the $\mathrm{ZFC}$-provable
principles of $\sigma$-centered forcing are exactly $\mathsf{S4.2}$.
\end{thm}
\begin{proof}
If $\mathrm{ZFC}$ is consistent then we can obtain the model $W$ and by propositions \ref{prop:buttons} and \ref{prop:n-switch} we obtain the buttons and $n$-switches required for using theorerm \ref{main-labeling}. So the $\mathrm{ZFC}$-provable principles of $\sigma$-centered forcing are contained in $\mathsf{S4.2}$, and by theorem \ref{thm:contained in} we get equality.
\end{proof}

\section{Generalizations and open questions}
\subsection{Forcing over $L$}\label{overL}
In  most of the calculations of the modal logic of a certain set-theoretic construction, the upper bound was obtained using control statement over the constructible universe $L$ (e.g. all the calculations in \cite{StC}). A notable exception is the upper bound for the modal logic of \emph{grounds} obtained in \cite{MUD} using Reitz's model, which is a class-forcing extension of $L$.
Similarly, for our construction we also first had to build a special model of $\mathrm{ZFC}$ -- a set-forciong extension of $L$, over which we could obtain the desired control statements. Whether we could have avoided this and work over $L$ itself remains an open question, which leads to the following more exact one:
\begin{que}\label{overL?}
	What is the modal logic of $\sigma$-centered forcing over $L$?
\end{que}
This question relates to a second line of inquiry introduced in \cite{MLF} -- the calculation of the valid principles of forcing over a specific model. It the case of all forcing notions, it has been shown in \cite{MLF} that these principles always contain $\mathsf{S4.2}$ and are contained in $\mathsf{S5}$, where both bounds can be realized. Models having other validities have been recently announced by Block and Hamkins (see discussion in \cite[pg. 32]{Piribauer}). It has also been shown (see \cite{Hamkins-Linnebo, Piribauer}) that any model has a ground whose valid principles of forcing are $\mathsf{S4.2}$. However, in the case of $\sigma$-centered forcing we have only limited results: $\mathsf{S4.2}$ is still a lower bound over any model, as this class is reflexive, transitive and persistent. As for an upper bound -- note that any model satisfying the assumptions of proposition \ref{prop:n-switch}, $L$ in particular, has $n$-switches for $\sigma$-centered forcing for every $n$. Such $n$-switches can be used, as in \cite[theorem 10]{StC}, to show that the valid principles of $\sigma$-centered forcing for such a model are contained in the modal theory $\mathsf{S5}$. So we do have:
\begin{prop}
	The modal logic of $\sigma$-centered forcing over $L$ (or any model satisfying  the assumptions of proposition \ref{prop:n-switch}) is between $\mathsf{S4.2}$ and $\mathsf{S5}$.
\end{prop}
However in the case of $L$ we do not expect to be able to raise the upper bound to $\mathsf{S5}$, as not even axiom $\mathrm{.3}$: 
\[
\left(\pos\fii\land\pos\psi\right)\to\pos\left[\left(\pos\fii\land\psi\right)\lor \left(\fii\land\pos\psi\right)\right]
\] 
 which corresponds to the linearity of the forcing class (see \cite{StC}), holds over $L$: by a result of B{\l}aszczyk and Shelah \cite{blaszczyk-shelah}, in $L$ there is a $\sigma$-centered forcing notion which does not add a Cohen real. So let $\p$ be the $L$-least such forcing notion, consider the statements:
\begin{align*}
\fii'&= \text{``} \text{There is a Cohen real over } L \text{''}\\
\psi'&= \text{``} \text{There is a } \p \text{-generic filter over }L \text{''}
\end{align*}
and set $\fii =\fii'\land\neg\psi'$, $\psi=\psi'\land\neg\fii'$. So clearly $L\vDash\pos\fii\land\pos\psi$ but $\left(\pos\fii\land\psi\right)\lor \left(\fii\land\pos\psi\right)$ is not forceable since a generic for one of these forcings will stay generic in further extensions. Hence question \ref{overL?} is still open, as well as the more general question:
\begin{que}
	What modal theories can be realized as the valid principles of $\sigma$-centered forcing over some model of $\mathrm{ZFC}$?
\end{que}

\subsection{Related forcing classes}\label{related classes}
Throughout this work, we have focused on $\sigma$-centered forcing
notions. However, by examining the proofs, one can see that we have
not used the full strength of $\sigma$-centeredness. To obtain the
lower bound, we used the reflexivity, transitivity and persistence
of $\sigma$-centered posets. And to obtain the upper bound, we defined labelings using two main ingredients -- the posets constructed section \ref{subsec:Buttons}, giving us the independent buttons, and the $n$-switch of proposition \ref{prop:n-switch}.
To work with the buttons, we also required that all extensions of
$W$ will be appropriate. Assuming this, once we had an $n$-switch,
we did not use it's specific construction in defining the labeling.
So in fact we have the following: 
\begin{thm}
\label{thm:Gener1}Let $W$ the model constructed in section \ref{subsec:ground-model}
and $\Gamma$ a class of forcing notions with the following properties:

\begin{enumerate}
\item $\Gamma$ is reflexive, transitive and persistent;
\item Every extension of~ $W$ by a $\Gamma$-forcing is appropriate;
\item All posets constructed in section \ref{subsec:Buttons} are in $\Gamma$;
\item There is an $n$-switch for $\Gamma$-forcing over $W$ for any $n$.
\end{enumerate}
Then $\mathsf{MLF}\left(\Gamma\right)=\mathsf{S4.2}$.
\end{thm}
Now let's see what was needed to obtain the $n$-switch of proposition
\ref{prop:n-switch}. We relied heavily on the c.c.c of all posets
in question; we used all posets coding subsets of $\omega_{1}$, as
well as products of at most $\aleph_{1}$ of them; we relied on the
fact that $\sigma$-centered posets cannot enlarge $2^{\aleph_{0}}$
or $2^{\aleph_{1}}$; we used the fact that there were (in $W$) only
$\aleph_{2}$ $\sigma$-centered posets up to equivalence, and that
they were all already in the smaller model $Z$. So, this construction
can be carried with any class of forcing notions satisfying these
requirements. To conclude:
\begin{thm}
\label{thm:Gener2}Let $\Gamma$ be a class of forcing notions with
the following properties:

\begin{enumerate}
\item $\Gamma$ is reflexive, transitive and persistent;
\item Every extension of $W$ by a $\Gamma$-forcing is appropriate;
\item All posets constructed in section \ref{subsec:Buttons} are in $\Gamma$;
\item[(4.1)] Each poset in $\Gamma$ has the c.c.c, and does not enlarge $2^{\aleph_{0}}$
or $2^{\aleph_{1}}$;
\item[(4.2)] $\left|\Gamma\right|\leq2^{2^{\aleph_{0}}}$ (where the size is measured
up to equivalence of forcing);
\item[(4.3)] $\Gamma^{W}\con Z$;
\item[(4.4)] All posets which are used to code subsets of $\omega_{1}$, and products
of at most $\aleph_{1}$ of them, are in $\Gamma$.
\end{enumerate}
Then $\mathsf{MLF}\left(\Gamma\right)=\mathsf{S4.2}$.
\end{thm}
\begin{rem*}
Conditions 3 and 4.4 will hold for any class containing
every $\sigma$-centered forcing. 
\end{rem*}
\begin{defn}
	A subset $C\con\p$ is called $n$\emph{-linked} if any $n$ elements
	of $C$ are compatible, i.e. have a common extension (perhaps not
	in $C$ itself). $2$-linked is also called simply \emph{linked}.
	A poset is called $\sigma$\emph{-$n$-linked }if it is the union
	of $\omega$ many $n$-linked subsets. Again, $\sigma$-linked means
	$\sigma$-2-linked.
\end{defn}
It is clear that we have the following implications:
\[
\mbox{\ensuremath{\sigma}-centered \,\ensuremath{\rightarrow\,\,}\ensuremath{\sigma}-\ensuremath{n}-linked for every \ensuremath{n}\,\,\ensuremath{\rightarrow...\rightarrow\,\,}\ensuremath{\sigma}-\ensuremath{n}-linked \,\ensuremath{\rightarrow\,\,}\ensuremath{\sigma}-linked}
\]
and it is known that the other directions do not hold (cf. \cite{weak-var}).
\begin{cor}
	Let $\Gamma$ be either the class of all $\sigma$-$n$-linked posets
	(for some fixed $n$), or the class of all posets which are\textup{
	}$\sigma$-$n$-linked for every $n$. Then $\mathsf{MLF}\left(\Gamma\right)=\mathsf{S4.2}$.
\end{cor}
\begin{proof}
	
Lemmas \ref{s.c.properties} and \ref{lem:small} hold for these classes as well, so they are all transitive, preserve cardinals and the continuum function, and have size at most $2^{2^{\aleph_{0}}}$ (up to equivalence). These classes are also clearly reflexive and persistent, and they contain the class of all $\sigma$-centered forcings, so they satisfy all the conditions of theorem \ref{thm:Gener2}.
\end{proof}

Parallel to this hierarchy of properties, we can define the following
hierarchy (cf. \cite{weak-var}):
\begin{defn}

\begin{enumerate}
\item Given $n\in\omega$, $\p$ has property $\mathrm{K}_{n}$ if
every $A\in\left[\p\right]^{\aleph_{1}}$ contains an uncountable
$n$-linked subset. $\mathrm{K}_{2}$ is also called the \emph{Knaster
property}.
\item $\p$ has \emph{pre-caliber $\omega_{1}$} if every $A\in\left[\p\right]^{\aleph_{1}}$
contains an uncountable centered subset.
\end{enumerate}
\end{defn}
Note that pre-caliber $\omega_{1}$ implies property $\mathrm{K}_{n}$,
and $\mathrm{K}_{n}$ implies $\mathrm{K}_{m}$ for $m\leq n$. So
these form a hierarchy of properties. Furthermore, if $\p$ is $\sigma$-centered
then it has pre-caliber $\omega_{1}$, and if it is $\sigma$-$n$-linked
then it has property $\mathrm{K}_{n}$. So we get the following implications:
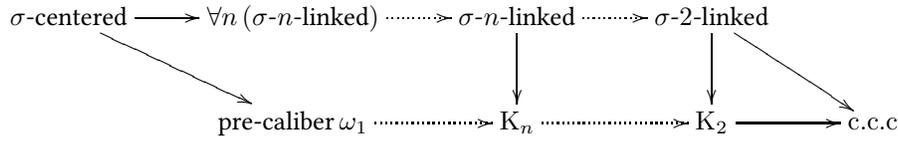
\[
	\xymatrix{\sigma\text{-centered}\ar[r]\ar[dr] 
	& \forall n\left(\sigma\text{-}n\text{-linked}\right)\ar@{.>}[r]
	& \sigma\text{-}n\text{-linked}\ar@{.>}[r]\ar[d]  
	& \sigma\text{-}2\text{-linked}\ar[d]
	\\
	&\text{pre-caliber}\,\omega_{1}\ar@{.>}[r]
	& \mathrm{K}_{n}\ar@{.>}[r]
	& \mathrm{K}_{2}\ar[r]
	& \mathrm{c.c.c}
}
\]

Let $\mathsf{P}$ be either pre-caliber $\omega_{1}$ or $\mathrm{K}_{n}$	for some $n$.
\begin{cor}
 Let $\Gamma_{\mathsf{P_{<\delta}}}$ the class of all $\mathsf{P}$-forcing
	notions of size $<\delta$, for some regular $\delta>2^{\aleph_0^{\aleph_0}} $. Then $\mathsf{MLF}\left(\Gamma_{\mathsf{P_{<\delta}}}\right)=\mathsf{S4.2}$.
\end{cor}
\begin{proof}
We verify the conditions of theorem \ref{thm:Gener1}. It is standard to verify that $\Gamma_{\mathsf{P_{<\delta}}}$ is reflexive, transitive and persistent (hence directed), so condition (1) holds.
Note that the coding of the reals $a_{\alpha,i}$
can be started as high as we want, so by limiting ourselves to forcing notions of a bounded size, we can do this coding somewhere high enough that will not be affected by these forcings, and therefore obtain condition (2) (note that by c.c.c cardinals are not changed).
Since $\delta>2^{\aleph_0^{\aleph_0}}$ these classes contains all $\sigma$-centered forcings, so condition (3) holds.

Since these forcings do not preserve the continuum,
we cannot obtain an $n$-switch for these classes using the same methods. However, note that for every $I$, the poset $\mathrm{Fn}\left(I,2\right)$
consisting of finite functions from $I$ to $2$, ordered by reverse
inclusion, has pre-caliber $\omega_{1}$ (cf. \cite[pg. 181]{KunenNew}), and forces $2^{\aleph_0}>|I|$, which by c.c.c is a button for $\mathsf{P}$-forcing. Therefore
the statements $\langle2^{\aleph_{0}}\geq\aleph_{\alpha}\mid\alpha<\delta\rangle$ form a strong ratchet, and we can use it to construct an $n$-switch
as in lemma \ref{lem:n-switch}, to obtain condition (4). So all the conditions of theorem
\ref{thm:Gener1} can be met, giving us the result.
\end{proof}

Note that bounding the size of the forcings was essential, since otherwise the classes contain all posets of the form $\mathrm{Fn}\left(I,2\right)$, so extensions of $W$ using such forcings may not be appropriate. We can however get some result on the full classes of $\mathsf{P}$-forcings:
\begin{thm} \label{thm:allP}
 Let $\Gamma_{\mathsf{P}} $ be the class of all $P$-forcings. Then $\mathsf{S4.2}\subseteq\mathsf{MLF}\left(\Gamma_{\mathsf{P}}\right)\subseteq\mathsf{S4.3}$. 
\end{thm}
\begin{proof}
	As we noted, these classes are reflexive, transitive and persistent, hence the left inclusion. Now, the statements $ r(\alpha)=2^{\aleph_{0}}\geq\aleph_{\alpha}$ for any $ \alpha \in \mathrm{Ord}$ such that $\mathrm{cf}(\alpha)>\omega$ form a long ratchet over $ L $ in the sense of \cite{StC}: these are indeed statements obtained by a single formula using parameter $ \alpha $;
	each of them is a pure button for every class of c.c.c forcings, since once $ 2^{\aleph_{0}}\geq\aleph_{\alpha} $ is true, it cannot be made false without collapsing cardinals; 
	clearly $ r(\alpha)\to r(\beta) $ for any $ \beta<\alpha $; 
	we can push $ r(\alpha) $ without pushing any further $ r(\beta) $ using the c.c.c poset $ \mathrm{Fn}(\aleph_\alpha\times\omega,2) $ of all finite functions from $\aleph_\alpha\times\omega$ to $\{0,1\}$, which forces $ 2^{\aleph_{0}}=\aleph_{\alpha} $ (whenever $ \mathrm{cf}(\alpha)>\omega $, see \cite[thm. IV.3.13]{KunenNew}. Note that the requirement there is $\aleph_\alpha^{\aleph_0}=\aleph_\alpha$, which is implied by $ \mathrm{cf}(\alpha)>\omega $ so will not change by c.c.c forcing);
	and clearly no forcing extension can satisfy every $ r(\alpha) $.
	So by \cite[theorem 12]{StC} (theorem \ref{thm:ratchet->S4.3}) we get the inclusion on the right. 
\end{proof}

To get an exact result would require a different method, so the following is open:
\begin{que}
Let $\mathsf{P}$ be either pre-caliber $\omega_{1}$
or $\mathrm{K}_{n}$ for some $n$ . What is the modal logic of all
$\mathsf{P}$-forcing notions?
\end{que}
Finally, the only property in the above diagram we did not discuss
yet is c.c.c.
\begin{que}
What is the modal logic of all c.c.c forcing notions?
\end{que}
This natural question was already raised in \cite{MLF}. The difficulty
in answering it is that the class of all c.c.c forcing notions is
\emph{not} directed, so it does not contain $\mathsf{S4.2}$.
It is reflexive and transitive, so Hamkins and L\"{o}we conjectured
that the answer is $\mathsf{S4}$. To prove this, one would probably
need to find a labeling for models based on trees, as the class of
all trees is a class of simple frames characterizing $\mathsf{S4}$.
It should be mentioned that
in \cite{Inamdar}, a labeling of models based frames which are ``spiked
pre-Boolean algebras'' (denoted $\mathsf{S4sBA}$, cf. \cite{Inamdar} for exact definition)
was provided, thus establishing an upper bound which is strictly between
$\mathsf{S4}$ and $\mathsf{S4.2}$. However it is not known whether
this modal theory is finitely axiomatizable, so it is not yet clear
whether this can be shown to be a lower bound as well by the current methods. So, this question remains open.

To conclude, we list the current knowledge concerning all the classes in the diagram.
\begin{thm} \
	\begin{enumerate}
	\item The modal logic of the following classes of forcing is exactly $\mathsf{S4.2}$:
	\begin{enumerate}
		\item $\sigma$-centered;
		\item $\sigma$-$n$-linked for all $n$;
		\item $\sigma$-$n$-linked for some $n$;
		\item Pre-caliber $\omega_1$ of bounded size;
		\item $\mathrm{K}_n$ (for some $n$) of bounded size.
	\end{enumerate}
	\item The modal logic of the following classes contains $\mathsf{S4.2}$ and is contained in $\mathsf{S4.3}$:
	
	\begin{enumerate}
		\item Pre-caliber $\omega_1$;
		\item $\mathrm{K}_n$ for some $n$.
	\end{enumerate}
	\item The modal logic of c.c.c forcing contains $\mathsf{S4}$ and is contained in $\mathsf{S4sBA}$.
	
\end{enumerate}
\end{thm}
\bibliographystyle{plain}
\bibliography{Thesis}

\end{document}